\documentclass{amsart}

\usepackage{multirow}
\usepackage{hyperref}
\usepackage[labelformat=simple]{subcaption}
\usepackage[ruled]{algorithm2e} 
\usepackage{tikz}
\usetikzlibrary{shapes,arrows}
\usepackage{pgfplots}
\pgfplotsset{compat=newest}
\usetikzlibrary{decorations.markings}

\newtheorem{thm}{Theorem}[section]

\theoremstyle{definition}
\newtheorem{defn}[thm]{Definition}
\theoremstyle{remark}
\newtheorem{rem}[thm]{Remark}
\numberwithin{equation}{section}
\newtheorem{ex}[thm]{Example}
\newtheorem{proper}{Property}

\def\R{\mathbb{R}}
\def\B{\mathbb{B}}

\def\ST{\mathcal{ST}}
\newcommand{\dsum}{\displaystyle\sum}
\newcommand{\dmin}{\displaystyle\min}
\newcommand{\dmax}{\displaystyle\max}

\title{Minimum Spanning Trees with neighborhoods}%

\author{V\'ictor Blanco}
\address{Dpt. Quant. Methods for Economics \& Business, Universidad de Granada}
\email{vblanco@ugr.es}

\author{Elena Fern\'andez}
\address{Dpt. Statistics \& OR, Universitat Polit\`ecnica de Catalunya}
\email{e.fernandez@upc.edu}

\author[us]{Justo Puerto}
\address{Dpt. Statistics \& OR, Universidad de Sevilla}
\email{puerto@us.es}

\begin{document}

\keywords{Minimum Spanning Trees, Neighborhoods, Mixed Integer Non Linear Programming, Second Order Cone Programming.}
\subjclass[2010]{90C27, 05C05, 90C11, 90C30.}

\maketitle

\begin{abstract}
This paper studies Minimum Spanning Trees under incomplete information for its vertices. We assume that no information is available on the precise placement of vertices so that it is only known that vertices belong to some neighborhoods that are second order cone representable and distances are measured with a $\ell_q$-norm. Two mixed integer non linear mathematical programming formulations are presented, based on alternative representations of subtour elimination constraints. A solution scheme is also proposed, resulting from a reformulation suitable for a Benders-like decomposition, which is embedded within an exact branch-and-cut framework. Furthermore, a mathheuristic is developed, which alternates in solving convex subproblems in different solution spaces, and is able to solve larger instances. The results of  extensive computational experiments are reported and analyzed.
\end{abstract}

\section{Introduction}\label{sec:Intro}
Nowadays Combinatorial Optimization (CO) lies in the heart of multiple applications in the field of Operations Research. Many such applications can be formulated as optimization problems defined on graphs where  some particular structure is sought satisfying some optimality property. Traditionally this type of problems assumed implicitly the exact knowledge of all input elements, and, in particular, of the precise position of vertices and edges. Nevertheless, this assumption does not always hold, as uncertainty, lack of information, or some other factors may affect the relative position of the elements of the input graph. Hence,  new tools are required to give adequate answers to these challenges, which have been often ignored by standard CO tools.

A matter that, in this context, has attracted the interest of researchers over the last years is the solution of certain CO problems when the exact position of the vertices of the underlying graph is not known with certainty. If probabilistic information is available, then stochastic programming tools can be used, and optimization over expected values carried out.  Moreover, even under the assumption of incomplete information one could use a uniform distribution and still apply such an approach. However, the use of probabilistic information and allowing to consider all possible locations for the vertices is not always suitable.
For instance, when for each point of the input graph, a unique representative associated with it must be determined.
Scanning the related literature one can find papers applying both methodologies. Examples of stochastic approaches are for instance \cite{BertsimasHowell93} or  \cite{frank}. Examples of the second type of approach arise in variants of the traveling salesman problem (TSP), minimum spanning tree (MST), or facility location problems that deal with \textit{demand regions} instead of \textit{demand points} (see \cite{AH94,Brimberg,Cooper_JRS78,Dror,Juel_OR81,NPR03,Yang07}, to mention just a few).

A relevant common question raised by the latter class of problems is how to model and solve optimization problems on graphs when vertices are not points but regions in a given domain. The above mentioned case of the TSP, first introduced by Arkin and Hassin \cite{AH94}, has been addressed recently by a number of authors. It generalizes the Euclidean TSP and the group Steiner tree problem, and has applications in VLSI-design and other routing problems,  in which only imprecise information of the positions of the vertices is available. Several inapproximability results and approximation algorithms have been developed for particular cases. The case of the spanning tree problem with neighborhoods (MSTN) was first addressed by Yang et al. \cite{Yang07}, who proved that the general case of the problem in the plane is NP-hard (result also reproved by L\"offler and van Kreveld in  \cite{loffler-kreveld}), and gave several approximation algorithms and a PTAS for the particular case of disjoint unit disks in the plane. Some extensions considering the maximization of the weights are studied in Dorrigiv et. al \cite{Dorrigiv}.  In particular, they proved the non existence of FPTAS for MSTN, for general disjoint disks, in the planar Euclidean case. Disser et. al \cite{Disser}  consider the shortest path problem and the rectilinear MSTN, and give some approximability results. To the best of our knowledge, Gentilini et al \cite{Gentilini} are the first authors to propose an exact Mixed Integer Non Linear Programming (MINLP) formulation for  the TSP with neighborhoods, but we are not aware of any MINLP for the MSTN.

Our goal in this paper is to develop MINLP formulations and solution methods for the MSTN.
We first present  two MINLP formulations that allow to solve medium size MSTN planar and 3$D$ Euclidean instances with up to 20 vertices, for neighborhoods of varying radii using an on-the-shelf solver. Furthermore, we develop an effective branch-and-cut strategy, based on a generalized Benders decomposition  \cite{Benders,geoffrion}, and compare its performance with that of the solver for the proposed formulations. For this we present an alternative formulation for the MSTN, in which the master problem consists in finding an MST with costs derived from a continuous non linear (slave) subproblem, and we develop the expression and separation of the cuts that are added in the solution algorithm.
Given that both the solver (for the two MINLP formulations) and the exact branch-and-cut algorithm can be too demanding, in terms of their computing times, we have also developed an effective and efficient mathheuristic. The mathheuristic stems from the observation that the subproblems defined in the solution spaces of each of the two main sets of variables are convex (so they can be solved very efficiently); it alternates in solving subproblems in each of these solution spaces.\\

The paper is organized as follows. Section \ref{sec:Definition} is devoted to introduce the MSTN and to state a generic formulation. In Section \ref{sec:Formulations} we present and compare two MINLP formulations for the MSTN, based on alternative representations of the spanning trees polytope. Section \ref{sec:Benders} develops the exact branch-and-cut algorithm, based on  a Benders-like decomposition scheme: we define the master and the non linear subproblem, and derive the cuts and their separation. In Section \ref{sec:Experiments} we first compare the performance of the on-the-shelf solver with the two MINLP formulations, and then we report the numerical results obtained with the exact  row-generation algorithm. The  mathheuristic is presented in Section \ref{sec:Mathheuristic}, where we also give the numerical results that it produces. The paper ends with some concluding remarks and our list of references.

 \section{Minimum Spanning Trees with Neighborhoods\label{sec:Definition}}

Let $G=(V,E)$ be a connected undirected graph, whose vertices are embedded in $\R^d$, i.e., $v \in \R^d$ for all $v \in V$. Associated with each vertex $v\in V$, let $\mathcal{N}_v\subseteq \R^d$ denote a convex set containing $v$ in its interior. Let also $\|\cdot\|$  denote a given norm.

Feasible solutions to the Minimum Spanning Tree with Neighborhoods (MSTN) problem consist of  a set of points, $Y^*=\{y_v \in \mathcal{N}_v\mid v \in V\}$, together with a spanning tree $T^*$ on the graph $G^*=(Y^*, E^*)$, with edge set $E^*=\{\{y_v,y_w\}: \{v,w\} \in E\}$. Edges lengths are given by  the norm-based distance between the selected points relative to $\|\cdot\|$, i.e.:
$$
d(y_v, y_w) = \|y_v-y_w\|, \quad \mbox{ for all } \{y_v, y_w\} \in E^*.
$$
 The overall cost of $(Y^*,T^*)$ is therefore
$$
d(T^*) = \dsum_{e=\{y_v,y_w\} \in T^*} d(y_v,y_w).
$$
The MSTN is to find a feasible solution, $(Y^*,T^*)$, of minimum total cost.\\

Particular cases of the MSTN have been studied in the literature for planar graphs. Disser et. al \cite{Disser} studied the case  when the sets $\mathcal{N}_v$ are
rectilinear neighborhoods centered at $v\in V$. Dorrigiv et. al \cite{Dorrigiv} addressed the problem when the sets $\mathcal{N}_v$ are disjoint Euclidean disks. Both referenced works study the complexity of the considered problems but do not attempt to develop MINLP formulations or solution methods for it.

In this paper, we consider the general case where the  graph $G$ is embedded in $\R^d$. Even if our developments can be extended to generic convex sets, we focus on the case where $\mathcal{N}_v$ is second order cone (SOC) representable \cite{vandenberghe}. The main reason for this is that state-of-the-art solvers incorporate mixed integer non-linear implementations of SOC constraints. Such a modeling assumption could be readily overcome if on-the-shelf solvers incorporated more general tools to deal with convex sets.

Observe that SOC representable neighborhoods allow to model, as a particular case, centered balls of a given radius $r_v$, associated with the standard $\ell_p$-norm with $p \in [1,\infty]$ in $\R^d$, that we denote by $\|\cdot\|_p$, i.e., neighborhoods in the form $\mathcal{N}_v = \{x \in \R^d: \|x-v\|_p \leq r_v\}$, where

$$
\|z\|_p = \left\{\begin{array}{cl}
\left(\dsum_{k=1}^d |z_k|^p\right)^{\frac{1}{p}} & \mbox{if $p<\infty$}\\
\dmax_{k\in \{1, \ldots, d\}} |z_k| & \mbox{if $p=\infty$}\end{array}\right..
$$

The reader is referred to \cite{BPE14} for further details on the SOC constraints that allow to represent (as intersections of second order cone and/or rotated second order cone constraints) such norm-based neighborhoods. Indeed, we can also easily handle neighborhoods defined as bounded polyhedra in $\R^d$, as well as intersections of polyhedra and balls. Hence, more sophisticated convex neighborhoods can be suitably represented or approximated using elements from the above mentioned families of sets.

Two extreme situations that can be modeled within our framework are the following. If the neighborhood for each vertex $v\in V$ is  the singleton $\mathcal{N}_v=\{v\}$, then MSTN becomes the classical MST problem with edge lengths given by the norm-based distances between each pair of vertices. On the other hand, if the considered neighborhoods are big enough so that $\bigcap_{v\in V} \mathcal{N}_v \neq \emptyset$, then the problem reduces to finding a degenerate one-vertex tree and the solution to the MSTN is that vertex with cost $0$.

Throughout this paper we use the following notation:
\begin{itemize}
\item $\mathcal{ST}_G$ as the set of incidence vectors associated with spanning trees on $G$, i.e. $\mathcal{ST}_G=\{x \in \R^{|E|}_+: x \mbox{ is a spanning tree on } G\}$
\medskip
\item $\mathcal{Y} = \displaystyle\prod_{v\in V} \mathcal{N}_v$, where $\mathcal{N}_v$ is the neighborhood associated to vertex $v$, which contains the possible sets of vertices for the spanning trees of MSTN. 
\end{itemize}

Then, the MSTN can be stated as:

\begin{align}
\min & \dsum_{e \in E}d(y_v, y_w) x_e\tag{${\rm MSTN}$}\label{mstn}\\
\mbox{s.t. } & x \in \mathcal{ST}_G, \; y \in \mathcal{Y}.\nonumber
\end{align}

Several observations follow from the formulation above:
\begin{enumerate}\label{comment}
\item Fixing $x\in \mathcal{ST}_G$ in  MSTN results in a continuous SOC problem, which is well-known to be convex \cite{vandenberghe}. On the other hand, fixing $y \in \mathcal{Y}$ results in a standard MST problem. It is a well-known that MST admits continuous linear programming representations \cite{Edmonds,martin}. Thus, MSTN can be seen as a biconvex optimization problem, which is neither convex nor concave \cite{biconvex}.
 \item Due to the expression of its objective function,  \ref{mstn} is not separable, even if each of its sets of variables  $x$ and $y$ belong to convex domains in different spaces.
\item Since \ref{mstn} {\it combines} the above two subproblems, it is suitable to be represented as a MINLP.\\
\end{enumerate}

The following example illustrates the  MSTN.
\begin{ex}
\label{ex:1}
Let us consider a graph with eight vertices and 14 edges, $G=(V,E)$ embedded in $\R^2$. The graph $G$ and an Euclidean Minimum Spanning Tree for this graph are shown in Figure \ref{fig1:ex1}. 

\begin{figure}[h]
\centering
    \begin{subfigure}[b]{0.45\textwidth}
\begin{center}
\begin{tikzpicture}[scale=0.5]


\coordinate(X1) at (0,5);
\coordinate(X2) at (1,1);
\coordinate(X3) at (1,6);
\coordinate(X4) at (1,4);
\coordinate(X5) at (3.5,3);
\coordinate(X6) at (9,3);
\coordinate(X7) at (7.5,0);
\coordinate(X8) at (8,6);

\fill (X1) circle (4pt);
\fill (X2) circle (4pt);
\fill (X3) circle (4pt);
\fill (X4) circle (4pt);
\fill (X5) circle (4pt);
\fill (X6) circle (4pt);
\fill (X7) circle (4pt);
\fill (X8) circle (4pt);

\draw (X1)--(X3);
\draw (X1)--(X4);
\draw (X2)--(X4);
\draw (X4)--(X5);
\draw (X5)--(X7);
\draw (X6)--(X7);
\draw (X6)--(X8);

\draw[gray, dotted] (X1)--(X2);
\draw[gray, dotted] (X2)--(X3);
\draw[gray, dotted] (X2)--(X4);
\draw[gray, dotted] (X3)--(X5);
\draw[gray, dotted] (X2)--(X7);
\draw[gray, dotted] (X3)--(X8);
\draw[gray, dotted] (X5)--(X6);
\draw[gray, dotted] (X5)--(X7);
\draw[gray, dotted] (X5)--(X8);
\draw[gray, dotted] (X2)--(X5);

\end{tikzpicture}
\end{center}
\caption{Input graph $G$ and Euclidean MST (black lines).\label{fig1:ex1}}
\end{subfigure}
~
 \begin{subfigure}[b]{0.45\textwidth}
    \begin{center}
\begin{tikzpicture}[scale=0.5]


\coordinate(X1) at (0,5);
\coordinate(X2) at (1,1);
\coordinate(X3) at (1,6);
\coordinate(X4) at (1,4);
\coordinate(X5) at (3.5,3);
\coordinate(X6) at (9,3);
\coordinate(X7) at (7.5,0);
\coordinate(X8) at (8,6);

\draw (X1) circle (1.0);
\draw (X2) circle (0.6);
\draw (X3) circle (1);
\draw (X4) circle (0.6);
\draw (X5) circle (1.4);
\draw (X6) circle (2.4);
\draw (X7) circle (0.8);
\draw (X8) circle (1);

\fill (X1) circle (2pt);
\fill (X2) circle (2pt);
\fill (X3) circle (2pt);
\fill (X4) circle (2pt);
\fill (X5) circle (2pt);
\fill (X6) circle (2pt);
\fill (X7) circle (2pt);
\fill (X8) circle (2pt);

\draw[gray, dotted] (X1)--(X3);
\draw[gray, dotted] (X1)--(X4);
\draw[gray, dotted] (X2)--(X4);
\draw[gray, dotted] (X4)--(X5);
\draw[gray, dotted] (X5)--(X7);
\draw[gray, dotted] (X6)--(X7);
\draw[gray, dotted] (X6)--(X8);

\draw[gray, dotted] (X1)--(X2);
\draw[gray, dotted] (X2)--(X3);
\draw[gray, dotted] (X2)--(X4);
\draw[gray, dotted] (X3)--(X5);
\draw[gray, dotted] (X2)--(X7);
\draw[gray, dotted] (X3)--(X8);
\draw[gray, dotted] (X5)--(X6);
\draw[gray, dotted] (X5)--(X7);
\draw[gray, dotted] (X5)--(X8);
\draw[gray, dotted] (X2)--(X5);
\end{tikzpicture}
\end{center}
\caption{Neighborhoods of the vertices.\newline{\color{white}abcdefg}\label{fig2:ex1}}
\end{subfigure}
\caption{Data for Example \ref{ex:1}.\label{figura:ex1}}
\end{figure}
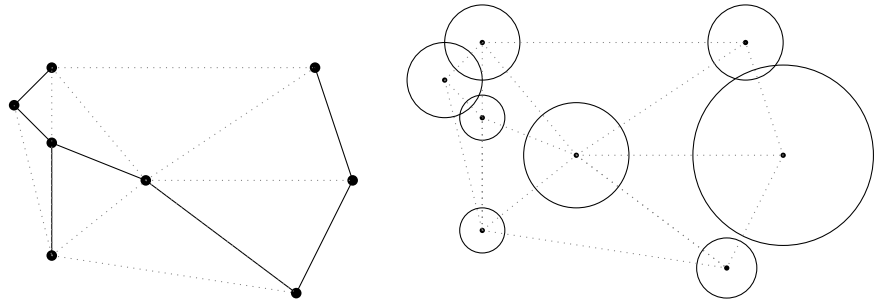

Figure \ref{fig2:ex1} 
shows the input graph together with the neighborhoods $N_v$ associated with the vertices $v \in V$. The neighborhoods are (Euclidean) balls centered at the original vertices, each of them with a different radius. 
 Figure \ref{fig3:ex1} shows an optimal MSTN solution: the location of the vertex selected  in each neighborhood, as well as the final spanning tree (both in gray).

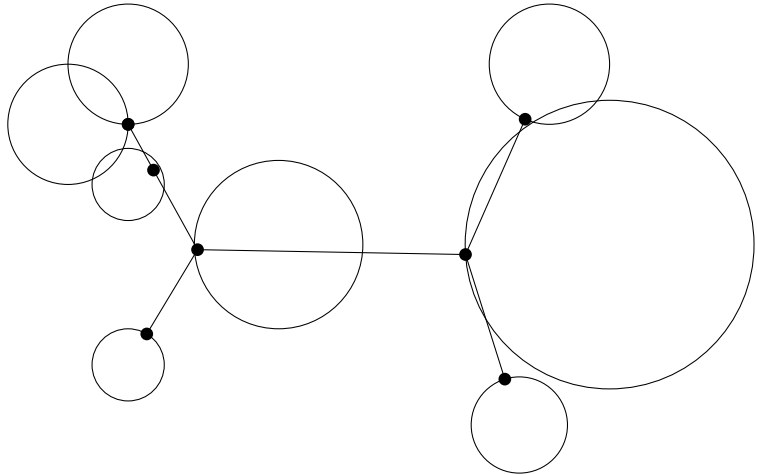
\begin{figure}[h]
\begin{center}
\begin{tikzpicture}[scale=0.8]


\coordinate(X1) at (0,5);
\coordinate(X2) at (1,1);
\coordinate(X3) at (1,6);
\coordinate(X4) at (1,4);
\coordinate(X5) at (3.5,3);
\coordinate(X6) at (9,3);
\coordinate(X7) at (7.5,0);
\coordinate(X8) at (8,6);

\draw (X1) circle (1.0);
\draw (X2) circle (0.6);
\draw (X3) circle (1);
\draw (X4) circle (0.6);
\draw (X5) circle (1.4);
\draw (X6) circle (2.4);
\draw (X7) circle (0.8);
\draw (X8) circle (1);


\coordinate(Y1) at (0.999997,4.999996);
\coordinate(Y2) at (1.309352,1.514104);
\coordinate(Y3) at (0.999997,5.000001);
\coordinate(Y4) at (1.421155,4.238826);
\coordinate(Y5) at (2.153065,2.915752);
\coordinate(Y6) at (6.605687,2.834897);
\coordinate(Y7) at (7.259147,0.762884);
\coordinate(Y8) at (7.596728,5.084919);


\fill (Y1) circle (3pt);
\fill (Y2) circle (3pt);
\fill (Y3) circle (3pt);
\fill (Y4) circle (3pt);
\fill (Y5) circle (3pt);
\fill (Y6) circle (3pt);
\fill (Y7) circle (3pt);
\fill (Y8) circle (3pt);
\draw (Y1)--(Y3);
\draw (Y1)--(Y4);
\draw (Y2)--(Y5);
\draw (Y4)--(Y5);
\draw (Y5)--(Y6);
\draw (Y6)--(Y7);
\draw (Y6)--(Y8);
\end{tikzpicture}
\end{center}
\caption{A MSTN for the data in Example \ref{ex:1}.\label{fig3:ex1}}
\end{figure}

Observe that the optimal spanning tree to the classical MST problem in the original input graph shown in Figure \ref{fig1:ex1}, with edges lengths given by the Euclidean distances between the initial vertices, is no longer valid for the MSTN. The reason is that the actual distances have been updated in order to consider the coordinates of the selected vertices, which are unknown beforehand. Note also that the structure of the original graph is somehow broken, since in the final solution some of the ``initial'' vertices are \textit{merged} into a single one (note that the MST in Figure \ref{fig3:ex1} has seven vertices while the original graph had eight). This is possible only when some of the neighborhoods have a non-empty intersection.\\

In Figure \ref{fig4:ex1} we show an optimal solution to the MSTN in the same input graph, for a different definition of the neighborhoods. Now they are defined as boxes in the form $\mathcal{N}_v = \{z\in\R^2: |z_k-v_k|\leq r_v, k=1, 2\}$.\\

\begin{figure}[h]
\begin{center}
\begin{tikzpicture}[scale=0.8]


\coordinate(X1) at (0,5);
\coordinate(X2) at (1,1);
\coordinate(X3) at (1,6);
\coordinate(X4) at (1,4);
\coordinate(X5) at (3.5,3);
\coordinate(X6) at (9,3);
\coordinate(X7) at (7.5,0);
\coordinate(X8) at (8,6);


\draw (-1,4)--(1,4)--(1,6)--(-1,6)--(-1,4);
\draw (0.4,0.4)--(1.6,0.4)--(1.6,1.6)--(0.4,1.6)--(0.4,0.4);
\draw (0,5)--(2,5)--(2,7)--(0,7)--(0,5);
\draw (0.4,3.4)--(1.6,3.4)--(1.6,4.6)--(0.4,4.6)--(0.4,3.4);
\draw (2.1,1.6)--(4.9,1.6)--(4.9,4.4)--(2.1,4.4)--(2.1,1.6);
\draw (6.6,0.6)--(11.4,0.6)--(11.4,5.4)--(6.6,5.4)--(6.6,0.6);
\draw (6.7,-0.8)--(8.3,-0.8)--(8.3,0.8)--(6.7,0.8)--(6.7,-0.8);
\draw (7,5)--(9,5)--(9,7)--(7,7)--(7,5);

\coordinate(Y1) at (1,5);
\coordinate(Y2) at (1.6,1.6);
\coordinate(Y4) at (1.529717,4.055101);
\coordinate(Y5) at (2.272081,2.729865);
\coordinate(Y6) at (6.6,2.673494);
\coordinate(Y7) at (6.7,0.8);
\coordinate(Y8) at (7,5);



\draw (Y2)--(Y5);
\draw (Y1)--(Y4);
\draw (Y4)--(Y5);
\draw (Y5)--(Y6);
\draw (Y6)--(Y7);
\draw (Y6)--(Y8);

\fill (Y1) circle (3pt);
\fill (Y2) circle (3pt);
\fill (Y4) circle (3pt);
\fill (Y5) circle (3pt);
\fill (Y6) circle (3pt);
\fill (Y7) circle (3pt);
\fill (Y8) circle (3pt);

\end{tikzpicture}
\end{center}
\caption{A MSTN for the data in Example \ref{ex:1} for polyhedral neighborhoods.\label{fig4:ex1}}
\end{figure}
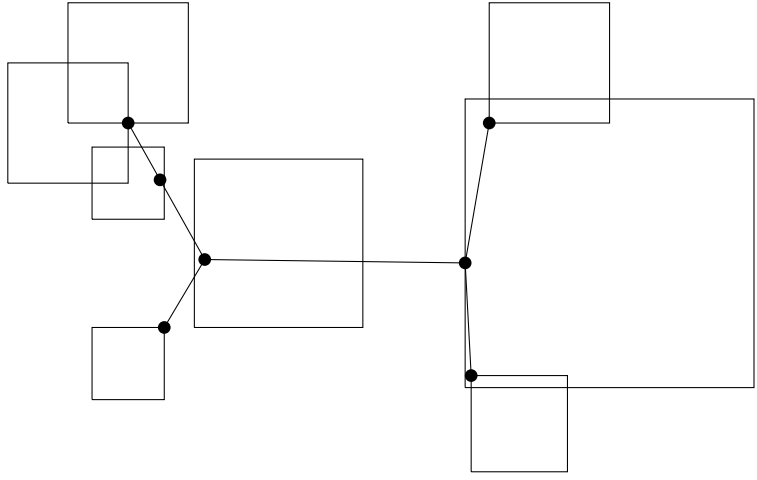

\end{ex}

As we see below, some of the properties of the standard MST extend to MSTN. In particular, the cut and cycle properties that allow reducing the dimensionality of MSTN by discarding edges that will not appear in an optimal solution as well as those edges that will appear in it. Before, we introduce the additional notation associated with each edge $e=\{v, w\}\in E$.\\
\begin{itemize}
\item $\widetilde{U}_e$ and $\widetilde{u}_e$ respectively denote the maximum and minimum distance between any pair of points in the neighborhoods of the end-vertices of $e$. That is, $\widetilde{U}_e = \max\{d(y_v, y_w): y_v\in \B_r^p(v), y_w\in \B_r^p(w)\}$ and $\widetilde{u}_e = \min\{d(y_v, y_w): y_v\in \B_r^p(v), y_w\in \B_r^p(w)\}$.\label{notation}
\end{itemize}

\begin{proper}\label{propiedad} $ $
\begin{itemize}
\item[(a)] Let $C$ be a cycle of $G=(V,E)$ and $e\in C$ such that  $\widetilde{u}_e  > \min_{e' \in E} \{\widetilde{U}_{e'}: e'\in C,e'\neq e\}$. Then, $e$ does not belong to a MSTN.
\item[(b)] Let $S \subset V$ and $(S, V\backslash S)=\left\{e=\{v,w\} \in E \mid v\in S \text{ and } w\in V\backslash S\right\}$ be its associated cutset. Let $e=\{v,w\} \in (S, V\backslash S)$  be such that $\widetilde{U}_e  < \min_{e' \in E} \{\widetilde{u}_{e^\prime}: e^\prime=\{v^\prime,w^\prime\}\in E, \; e'\neq e, v^\prime\in S, w^\prime \in V\backslash S\}$. Then, $e$ belongs to any MSTN.\\
\end{itemize}
\end{proper}
\begin{proof} $ $

\begin{itemize}
\item[(a)] Let $C$ be a cycle of $G=(V,E)$ and $e\in C$ such that $\widetilde{u}_e  > \min_{e \in E} \{\widetilde{U}_{e'}: e'\in C,e'\neq e\}$.\\
 Suppose, there is an MSTN of $G$, $T$ with $e\in T$. Then, for any other edge $e^\prime$ in the cycle $C$, the tree $T^\prime = T \cup \{e^\prime\}\backslash \{e\}$ satisfies that:
$$
d(T^\prime) \leq d(T) + \widetilde{U}_{e^\prime} -  \widetilde{u}_{e}  < d(T).
$$
Thus, the cost of $T^\prime$ is strictly smaller than the cost of $T$, contradicting the optimality of $T$. Hence $e$ will not appear in $T$.
\item[(b)] Let $T$ be a MSTN of $G$ with $e\not\in T$. Since $T$ is a tree, the unique cycle of $T \cup \{e\}$ contains both $e$ and the unique path in $G$ connecting $v$ and $w$, that does not contain $e$. Let $e^\prime$ the edge in such a path crossing the cut, i.e., $e^\prime=\{v^\prime, w^\prime\}$ with $v^\prime\in S$ and $w^\prime$ in $V\backslash S$. Then, $T^\prime = T \cup \{e\} \backslash \{e^\prime\}$ is a tree and such that
$$
d(T^\prime) \leq d(T) + \widetilde{U}_{e} -  \widetilde{u}_{e^\prime} < d(T),
$$
so $T^\prime$ has an overall distance smaller than $T$, contradicting its optimality. Hence $e$ will appear in $T$.
\end{itemize}
\end{proof}

\section{Mixed Integer Non Linear Programming Formulations }\label{sec:Formulations}

In this section we present alternative MINLP formulations for the MSTN that will be compared computationally in later sections. All formulations use the following sets of decision variables:

\begin{itemize}
\item Binary variables $x_e\in \{0, 1\}$, $e\in E$, to represent the edges of the spanning trees.
\item Continuous variables $y_v\in \mathcal{N}_v$, $v\in V$, to represent the point selected in each neighborhood.
\item Continuous variables $u_e\ge 0$, $e=\{v, w\}\in E$, to represent the distance $d(y_v, y_w)$ between the pairs of selected points.
\end{itemize}

Properties \ref{propiedad}(a) and (b) can be exploited in order to reduce the number of $x$ variables in the formulations. In particular, we  only need to define variables $x_e$ associated with edges that do not satisfy the condition \ref{propiedad}(a). On the other hand, we can set at value 1 all variables $x_e$ associated with edges that satisfy \ref{propiedad}(b).\\

Let $\mathcal{U}=\{u\in \R^{|E|}_+: u_e\geq d(y_v, y_w), \text{for all } e=\{v, w\}\in E, \text{ for some } y\in \mathcal{Y}\}$ denote implicitly the domain for the feasibility of the $u$ variables. Then, a generic bilinear formulation for MSTN is

\begin{align}
\min & \dsum_{e \in E}u_e x_e\tag{${\rm P_{xu}}$}\label{matheu0}\\
\text{s.t. } & x \in \mathcal{ST}_G, \; u \in \mathcal{U}.\nonumber
\end{align}

In the following we resort to McCormick's envelopes \cite{mccormick} for the linearization of the bilinear terms of the objective function. For this, we define  an additional set of continuous decision variables $\theta_e \ge 0$, $e\in E$ to represent the products $u_ex_e$. Then the linearization of the generic formulation \ref{matheu0} is:

\begin{align}
\min \qquad &  \Theta =\dsum_{e \in E} \theta_e\tag{${\rm RL-MSTN}$}\label{rl-mstn}\\
\text{s.t. } & \theta_e \geq u_e - \widetilde{U}_e (1-x_e), \quad \forall e \in E,\tag{${\rm LIN-Mc}$}\label{linz}\\
&x \in \mathcal{ST}_G, \quad  u \in \mathcal{U}, \quad  \theta_e \geq 0, e\in E.\nonumber
\end{align}

Furthermore, throughout we will describe the set $\mathcal{U}$ using the set of constraints
\begin{alignat}{2}
             & \|y_v - y_w\|\leq u_e,                    && \qquad \forall e=\{v, w\} \in E, \label{u1}\tag{${\rm U}_1$}\\
             & y \in \mathcal{Y},                      &&\label{u2}\tag{${\rm U}_2$}
\end{alignat}
which set the distance values and impose that the $y$ points belong to the appropriate neighborhoods, respectively. \\

Note that the above formulation (\ref{rl-mstn}) can be reinforced by adding the following valid inequalities: $\theta_e\ge \tilde u_e x_e$, for all $e\in E$.

The two formulations below differ from each other in the representation of subtour elimination constraints (SEC). One of them uses the classical  representation of \cite{Edmonds}, which consists of a family with an exponential number of inequalities. The other one   uses a compact formulation based on the well-known Miller-Tucker-Zemlin (MTZ) constraints \cite{mtz}. Despite having a weaker linear programming bound than the subtour elimination representation for the classical MST problem, we use this formulation since, in practice, it has given quite good results for other problems related to spanning trees \cite{marin,pozo}. Indeed, other compact representations could be used, like for instance, the one by Martin \cite{martin}. In our experience,  \cite{mtz} gives a good tradeoff between the number of variables it requires and the bounds it produces.

\subsection{MSTN formulation based on classical representation of SECs}

\begin{alignat}{2}
\min     \qquad     & \Theta =\dsum_{e \in E} \theta_e                  &&\tag{${\rm SEC-MSTN}$}\label{sec-mstn}\\
\text{s.t. } & \eqref{linz}, \eqref{u1},\eqref{u2},\nonumber\\
             & \dsum_{e \in E} x_e = |V|-1,              &&\label{st1}\tag{${\rm ST}_1$}\\
             & \dsum_{e=\{v,w\}: v, w \in S} x_e \leq |S|-1, && \forall S \subset V, \label{st2}\tag{${\rm ST}_2$}\\
             & u, \theta \in \R^{|E|}_+, y \in \R^{|V|\times d}, x \in && \{0,1\}^{|E|}.\label{dom}\tag{${\rm D_1}$}
\end{alignat}

Constraints \eqref{st1} impose that exactly $|V|-1$ edges are selected and subtours are prevented by \eqref{st2}.
\eqref{dom} define the domain of the variables.

As mentioned, the number of constraints in \eqref{st2} is exponential on $|V|$, so a separation procedure (e.g. max flow - min cut)   to certify whether a solution is feasible or otherwise, to provide a violated constraint, is needed to solve this formulation. This is avoided in the next formulation, which uses the MTZ compact representation of SECs \cite{mtz}.

\subsection{MSTN formulation based on Miller-Tucker-Zemlin}

The formulation based on the MTZ representation of SECs builds a tree rooted at an arbitrarily selected vertex where the arcs of the tree are oriented \textit{towards} the root.
In our case we set vertex 1 as the root of the trees. Associated with each edge $\{v, w\}\in E$ we define two additional binary decision variables, $z_{vw}$ and $z_{wv}$, to indicate whether or not  $(v, w)$ (resp. $(w, v)$) is used as a directed arc. The set of such arcs is denoted by $A$. As it is usual for the representation of the SEC constraints we use continuous variables $s_{v}$, $v\in V$, associated with the vertices. The MTZ-MSTN formulation is:

\begin{alignat}{2}
\min \qquad & \Theta =\dsum_{e \in E} \theta_e &&\tag{${\rm MTZ-MSTN}$}\label{mtz}\\
\text{s.t. } & \eqref{linz}, \eqref{u1},\eqref{u2},\nonumber\\
& x_e =  z_{uv} + z_{vu} ,&& \forall e = \{u,v\} \in E, \label{mtz1}\tag{${\rm MTZ}_1$}\\
& \dsum_{(v,1) \in \delta^-(1)} z_{v1} \geq 1, &&       \label{mtz2}\tag{${\rm MTZ}_2$}\\
& \dsum_{(v,w) \in \delta^-(u)} z_{vw} = 1,  && \forall v \in V\backslash\{1\}, \label{mtz3}\tag{${\rm MTZ}_3$}\\
& |V|z_{vw} + s_v - s_w \leq |V|-1, && \forall (v, w) \in A,\label{mtz4}\tag{${\rm MTZ}_4$}\\
& s_1=1; 2\leq s_u \leq |V|, && \forall u \in V\backslash\{1\}, \label{mtz5}\tag{${\rm MTZ}_5$}\\
& u, \theta \in \R^{|E|}_+, y \in \R^{|V|\times d}, x \in  \{0,1\}^{|E|},&&\tag{${\rm D_1}$}\\
& z \in \{0,1\}^{|E|}, s\in \R^{|V|}_+.&&\label{dom2}\tag{${\rm D_2}$}
\end{alignat}

The meaning of the new constraints is as follows. Constraints \eqref{mtz1} relate the \textit{edge} and  \textit{arc} decision variables. The connectivity with the root is guaranteed by \eqref{mtz2}-\eqref{mtz3}. Subtours are eliminated by \eqref{mtz4} -\eqref{mtz5}, where the later set appropriate bounds for the vertex variables $s$. The domain of the new variables is set by \eqref{dom2}.\\

As mentioned, the two formulations presented above use the \textit{norm} constraints \eqref{u1} and \eqref{u2} to represent both the distance measure for the edges and the neighborhoods. As we see below both sets of constraints
can also be handled by using either SOC or linear constraints. The following remarks show the explicit representation of some general cases of this type of constraints.\\

\begin{rem} [$\ell_q$-norm representation]\label{rem:1}
As shown in \cite[Lemma 3]{BPE14}, if the norm $\|\cdot\|$ is a $\ell_q$-norm with $q\in \mathbb{Q}$ and $q=\dfrac{r}{s}>1$ (with $gcd(r,s)=1$), then the constraints of the form $\|X-Y\|_{q}\leq Z$ as those of \eqref{u1} can be rewritten as the following set of inequalities:

\begin{minipage}{\textwidth}
\begin{minipage}{0.75\textwidth}
$$\hspace*{3cm}
\left.\begin{array}{ll}
Q_{k} + X_k - Y_{k}\ge 0,&\; k=1, \ldots, d, \\
Q_{k} - X_k + Y_{k}\ge 0,& \; k=1, \ldots, d,\\
(Q_{k})^{r} \leq (R_{k})^{s} Z^{r-s},&  k=1, \ldots, d,\\
\dsum_{k=1}^d R_{k} \leq Z,& \\
R_k \geq 0, & k=1, \ldots, d,\end{array}\right\}
$$
where for $k=1, \ldots, d$, $Q_k =|X_k - Y_k|$ and $R_k=|X_k-Y_k|^{q} Z^{-1/\rho}$, with $\rho=\frac{r}{r-s}$.\\
\end{minipage}
\begin{minipage}{0.2\textwidth}
\begin{equation}\label{in:norm}
\end{equation}
\end{minipage}
\end{minipage}
The above gives a representation of \eqref{u1} with a number of SOC inequalities
that is polynomial in the dimension d and q.
\end{rem}

\begin{rem}[Polyhedral norm representation]
When the  norm $\|\cdot\|$ is a polyhedral (or block) norm, a (linear) representation, much simpler   than the one given in Remark \ref{rem:1} is possible. Let $B^*$ be the unit ball of its dual norm and  ${\rm Ext}(B^*)$ the set of extreme points of $B^*$. The constraint $Z \geq \|X-Y\|$ is then equivalent to
$$
Z \geq e^t (X-Y), \; \forall e \in {\rm Ext}(B^*),
$$
where $e^t$ denotes the transpose of $e$.
\end{rem}

\subsection{Computational comparison of the two formulations}

We have performed a series of computational experiments in order to compare the performance of the two formulations \ref{sec-mstn} and \ref{mtz}, as well as to explore the limitations of each of them. For this we have generated several batteries of instances with different settings. We consider complete graphs with a number of vertices ranging in $[5,20]$,  and randomly generated coordinates in $\R^2$ and $\R^3$ ranging in $[0,100]$.   Distances are measured using the Euclidean norm and Euclidean balls are used as neighborhoods of the vertices. In addition, we consider four different scenarios for generating the radii to define the neighborhoods of each vertex in a given instance:
\begin{description}
\item[Small size Neighborhoods ($r=1$):] Radii randomly generated in $[0,5]$.
\item[Small-Medium size Neighborhoods ($r=2$):] Radii randomly generated in $[5,10]$.
\item[Medium-Large size Neighborhoods ($r=3$):] Radii randomly generated in $[10,15]$.
\item[Large size Neighborhoods ($r=4$):] Radii randomly generated in $[15,20]$.
\end{description}
The above four cases allow us to observe the performance of the formulations for neighborhoods of varying sizes and to analyze how these sizes affect the computation the MSTN in each case.
Finally, five different instances were generated for each combination of number of vertices and radii, both in the plane and in the $3D$-space. The generated data are available at \url{bit.ly/mstneigh}.

All the formulations were coded in C, and solved using Gurobi  6.5 \cite{gurobi} in a Mac OSX El Capitan with an Intel Core i7 processor at 3.3 GHz and 16GB of RAM. A time limit of 2 hours was set in all the experiments.

Tables \ref{table:1}-\ref{table:2} summarize the results of these experiments. In these tables the column \textit{CPU}, under the heading of each formulation, reports the average computing time (in seconds) to attain optimality.  Whenever the time limit of 2 hours is reached without certifying optimality, columns under \textit{GAP} report the average percentage deviation of the best solution found during the exploration with respect to the lower bound at termination. Columns under \textit{\#Nodes} report the average number of nodes explored in the branch-and-bound search, whereas column \textit{SEC} gives the average number of  constraints \eqref{st2} incorporated to formulation SEC-MSTN throughout the solution process. Finally, the last column in each block reports the percentage of instances optimally solved with each formulation.

Observe that the computing times required by SEC-MSTN are in most cases smaller than those required by MTZ-MSTN. Furthermore, some instances that could not be solved with MTZ-MSTN, were optimally solved with SEC-MSTN. In most of the cases where SEC-MSTN did not succeed, MTZ-MSTN  was also not able to solve the corresponding instance. Note that, for the instances with $n=20$, we only report the results  for the first scenario ($r=1$), since neither  SEC-MSTN nor MTZ-MSTN were able to solve any of such instances for $r\geq 2$. We would like to highlight that, even if the $3$-dimensional instances have a higher number of variables than the planar ones, the results, in terms of computing times, percentage deviations, and number of optimally solved instances are better for these instances than for the $2$-dimensional ones.  Observe that the difficulty of an instance is highly related to whether or not the neighborhoods have non-empty intersections; in such cases, the continuous relaxation tends to \textit{collapse} the vertices of intersecting neighborhoods into a single one, which is not necessarily an optimal strategy. This justifies the higher difficulty of planar instances since, with uniform randomly generated points and  given radii, the probability of intersection of neighborhoods is higher in case of the plane than in the space \cite{Dufour}.

\renewcommand{\tabcolsep}{0.1cm}

\begin{table}[hbpt]
  \caption{Results of \ref{mtz} and \ref{sec-mstn} for $\R^2$ instances.}
  {\scriptsize  \begin{tabular}{|r|r|rrrr|rrrrr|}\cline{3-11}
 \multicolumn{2}{c|}{} &\multicolumn{4}{c|}{ \ref{mtz}} & \multicolumn{5}{c|}{\ref{sec-mstn}} \\\hline
             $r$  & $n$      & CPU & \#Nodes & GAP & \%Solved & CPU & \#SECs & \#Nodes & GAP & \%Solved \\\hline
\multirow{12}{*}{1}        & 5     & 0.0652 & 5.40  &       & 100\% & 0.0250 & 3.40  & 9.00  &       & 100\% \\
          & 6     & 0.0965 & 7.60  &       & 100\% & 0.0334 & 6.40  & 21.20 &       & 100\% \\
          & 7     & 0.1403 & 84.60 &       & 100\% & 0.0456 & 9.60  & 54.00 &       & 100\% \\
          & 8     & 0.1917 & 201.60 &       & 100\% & 0.0677 & 9.20  & 41.40 &       & 100\% \\
          & 9     & 0.2592 & 37.60 &       & 100\% & 0.0826 & 29.60 & 76.00 &       & 100\% \\
          & 10    & 0.4843 & 434.80 &       & 100\% & 0.1318 & 64.60 & 241.40 &       & 100\% \\
          & 11    & 0.6472 & 568.20 &       & 100\% & 0.3922 & 123.80 & 552.60 &       & 100\% \\
          & 12    & 0.9159 & 712.00 &       & 100\% & 0.3083 & 156.40 & 547.80 &       & 100\% \\
          & 13    & 10.9525 & 3145.80 &       & 100\% & 1.1175 & 419.00 & 1314.80 &       & 100\% \\
          & 14    & 4.7581 & 4014.80 &       & 100\% & 1.1627 & 300.40 & 1043.60 &       & 100\% \\
          & 15    & 657.1666 & 41153.60 &       & 100\% & 444.5906 & 1474.20 & 17828.00 &       & 100\% \\
          & 20    & 2915.1011 & 110070.80 &       & 100\% & 840.0096 & 2431.20 & 32173.80 &       & 100\% \\\hline
 \multirow{11}{*}{2}           & 5     & 0.0820 & 47.00 &       & 100\% & 0.0263 & 7.40  & 54.60 &       & 100\% \\
          & 6     & 0.1226 & 44.10 &       & 100\% & 0.0451 & 11.90 & 84.80 &       & 100\% \\
          & 7     & 0.1571 & 123.20 &       & 100\% & 0.0582 & 18.60 & 95.60 &       & 100\% \\
          & 8     & 0.4895 & 480.80 &       & 100\% & 0.2000 & 98.40 & 457.40 &       & 100\% \\
          & 9     & 0.5531 & 415.80 &       & 100\% & 0.3984 & 128.40 & 666.20 &       & 100\% \\
          & 10    & 1.3820 & 915.40 &       & 100\% & 0.7600 & 174.40 & 1125.00 &       & 100\% \\
          & 11    & 1.6639 & 835.60 &       & 100\% & 1.2961 & 235.80 & 1050.20 &       & 100\% \\
          & 12    & 32.8139 & 12301.20 &       & 100\% & 8.2899 & 832.80 & 9301.60 &       & 100\% \\
          & 13    & 143.7873 & 16259.40 &       & 100\% & 9.7330 & 4685.40 & 68409.20 &       & 100\% \\
          & 14    & 1467.5540 & 44337.00 & 7.64\% & 80\% & 661.3465 & 3252.60 & 36310.60 &       & 100\% \\
          & 15    & 3428.0761 & 423135.80 & 4.97\% & 80\% & 3424.9741 & 15712.80 & 179939.00 & 6.29\% & 60\% \\\hline
\multirow{11}{*}{3}            & 5     & 0.0958 & 44.20 &       & 100\% & 0.0354 & 9.40  & 79.80 &       & 100\% \\
          & 7     & 0.2645 & 414.60 &       & 100\% & 0.2772 & 189.70 & 1133.40 &       & 100\% \\
          & 8     & 1.6716 & 2097.80 &       & 100\% & 1.1393 & 338.60 & 1894.20 &       & 100\% \\
          & 9     & 3.7345 & 3827.40 &       & 100\% & 3.8655 & 407.60 & 3515.40 &       & 100\% \\
          & 10    & 5.9807 & 3465.20 &       & 100\% & 3.8294 & 333.80 & 2426.20 &       & 100\% \\
          & 11    & 713.2283 & 172376.20 &       & 100\% & 976.5382 & 61128.20 & 363205.60 &       & 100\% \\
          & 12    & 1054.4171 & 479364.20 &       & 100\% & 2828.2251 & 97800.80 & 576762.00 &       & 100\% \\
          & 13    & 3323.6210 & 279362.20 & 13.45\% & 60\% & 4626.0085 & 116751.40 & 953914.60 & 20.98\% & 80\% \\
          & 14    & $>$7200 & 1385623.40 & 30.04\% & 0\% & $>$7200 & 27120.40 & 162667.60 & 38.07\% & 0\% \\
          & 15    & $>$7200 & 1473884.40 & 19.43\% & 0\% & $>$7200 & 87730.20 & 392951.00 & 23.65\% & 0\% \\\hline
 \multirow{11}{*}{4}           & 5     & 0.0886 & 33.20 &       & 100\% & 0.0288 & 4.80  & 47.40 &       & 100\% \\
          & 6     & 0.1688 & 307.20 &       & 100\% & 0.1797 & 95.80 & 709.20 &       & 100\% \\
          & 8     & 2.0333 & 1976.60 &       & 100\% & 1.1078 & 289.80 & 1562.40 &       & 100\% \\
          & 9     & 4.4483 & 4936.00 &       & 100\% & 9.3935 & 444.60 & 6657.20 &       & 100\% \\
          & 10    & 67.5709 & 33224.80 &       & 100\% & 194.9068 & 1224.20 & 28680.60 &       & 100\% \\
          & 11    & 469.3033 & 198141.80 &       & 100\% & 315.9130 & 6463.80 & 70995.60 &       & 100\% \\
          & 12    & 2471.0749 & 403914.60 & 6.45\% & 80\% & 822.4408 & 105361.40 & 906147.00 &       & 100\% \\
          & 13    & 4609.7707 & 874785.60 & 16.88\% & 40\% & 5134.5084 & 8477.00 & 163847.00 & 19.64\% & 40\% \\
          & 14    & $>$7200 & 807955.40 & 44.52\% & 0\% & $>$7200 & 37016.40 & 192311.20 & 51.26\% & 0\% \\
          & 15    & $>$7200 & 948641.60 & 34.07\% & 0\% & $>$7200 & 29946.80 & 168779.80 & 43.33\% & 0\% \\\hline
    \end{tabular}}
  \label{table:1}
  \end{table}

\begin{table}[hbpt]
  \caption{Results of  \ref{mtz} and \ref{sec-mstn} for $\R^3$ instances.}
  {\scriptsize  \begin{tabular}{|r|r|rrrr|rrrrr|}\cline{3-11}
 \multicolumn{2}{c|}{} &\multicolumn{4}{c|}{ \ref{mtz}} & \multicolumn{5}{c|}{\ref{sec-mstn}} \\\hline
             $r$  & $n$      & CPU & \#Nodes & GAP & \%Solved & CPU & \#SECs & \#Nodes & GAP & \%Solved \\\hline
  \multirow{12}{*}{1}          & 5     & 0.0677 & 3.60  &       & 100\% & 0.0282 & 2.40  & 17.20 &       & 100\% \\
          & 6     & 0.1049 & 11.80 &       & 100\% & 0.0429 & 3.00  & 14.00 &       & 100\% \\
          & 7     & 0.2137 & 24.40 &       & 100\% & 0.0694 & 5.60  & 24.80 &       & 100\% \\
          & 8     & 0.2439 & 52.40 &       & 100\% & 0.0813 & 6.20  & 38.40 &       & 100\% \\
          & 9     & 0.3733 & 166.80 &       & 100\% & 0.1298 & 13.40 & 127.40 &       & 100\% \\
          & 10    & 0.3803 & 56.20 &       & 100\% & 0.1442 & 34.00 & 127.40 &       & 100\% \\
          & 11    & 1.0249 & 281.40 &       & 100\% & 0.3568 & 27.60 & 336.20 &       & 100\% \\
          & 12    & 0.6932 & 235.20 &       & 100\% & 0.2772 & 62.00 & 225.00 &       & 100\% \\
          & 13    & 1.3241 & 763.40 &       & 100\% & 0.9351 & 113.60 & 819.60 &       & 100\% \\
          & 14    & 4.1596 & 1112.00 &       & 100\% & 2.6353 & 200.80 & 1164.60 &       & 100\% \\
          & 15    & 4.2952 & 1286.20 &       & 100\% & 2.5708 & 197.00 & 812.40 &       & 100\% \\
          & 20    & 67.5323 & 6555.20 &       & 100\% & 8.9617 & 372.20 & 1441.00 &       & 100\% \\\hline
    \multirow{11}{*}{2}        & 5     & 0.0983 & 12.40 &       & 100\% & 0.0431 & 6.80  & 37.40 &       & 100\% \\
          & 6     & 0.1479 & 27.40 &       & 100\% & 0.0497 & 4.70  & 35.30 &       & 100\% \\
          & 7     & 0.2058 & 51.80 &       & 100\% & 0.0770 & 9.20  & 55.80 &       & 100\% \\
          & 8     & 0.3084 & 211.40 &       & 100\% & 0.1645 & 49.80 & 263.00 &       & 100\% \\
          & 9     & 0.8943 & 382.00 &       & 100\% & 0.4596 & 86.20 & 593.80 &       & 100\% \\
          & 10    & 0.5047 & 170.60 &       & 100\% & 0.2185 & 50.60 & 267.80 &       & 100\% \\
          & 11    & 1.4917 & 653.40 &       & 100\% & 0.5416 & 134.00 & 679.60 &       & 100\% \\
          & 12    & 3.2860 & 1814.40 &       & 100\% & 5.4726 & 462.80 & 2440.20 &       & 100\% \\
          & 13    & 5.3095 & 1956.40 &       & 100\% & 5.6612 & 437.20 & 2344.40 &       & 100\% \\
          & 14    & 16.8888 & 4485.20 &       & 100\% & 13.0737 & 1108.60 & 9084.40 &       & 100\% \\
          & 15    & 100.5050 & 14664.20 &       & 100\% & 54.8965 & 1524.20 & 12674.20 &       & 100\% \\\hline
   \multirow{11}{*}{3}         & 5     & 0.1034 & 12.00 &       & 100\% & 0.0450 & 3.00  & 39.60 &       & 100\% \\
          & 7     & 0.2737 & 199.30 &       & 100\% & 0.1663 & 79.70 & 428.00 &       & 100\% \\
          & 8     & 1.0901 & 972.40 &       & 100\% & 1.6812 & 230.40 & 1323.80 &       & 100\% \\
          & 9     & 15.9457 & 3589.40 &       & 100\% & 2.0036 & 295.00 & 3520.80 &       & 100\% \\
          & 10    & 2.0609 & 1124.00 &       & 100\% & 2.2637 & 259.80 & 1459.20 &       & 100\% \\
          & 11    & 29.7077 & 5477.80 &       & 100\% & 34.5579 & 549.20 & 7713.00 &       & 100\% \\
          & 12    & 330.0074 & 19946.80 &       & 100\% & 531.3279 & 1580.20 & 20383.00 &       & 100\% \\
          & 13    & 1069.2640 & 37625.20 &       & 100\% & 668.1420 & 2349.60 & 30331.40 &       & 100\% \\
          & 14    & 3875.3014 & 152561.80 & 15.19\% & 60\% & 2519.3367 & 11488.00 & 112377.40 & 6.87\% & 80\% \\
          & 15    & 1001.7704 & 47758.80 &       & 100\% & 160.5466 & 4114.40 & 37114.80 &       & 100\% \\\hline
     \multirow{11}{*}{4}       & 5     & 0.0875 & 21.60 &       & 100\% & 0.0469 & 6.80  & 42.60 &       & 100\% \\
          & 6     & 0.2094 & 134.20 &       & 100\% & 0.1156 & 28.00 & 255.40 &       & 100\% \\
          & 8     & 0.8188 & 832.20 &       & 100\% & 1.1261 & 204.00 & 1188.60 &       & 100\% \\
          & 9     & 2.8822 & 2408.60 &       & 100\% & 1.7530 & 329.40 & 4937.60 &       & 100\% \\
          & 10    & 6.4525 & 3461.40 &       & 100\% & 7.0799 & 525.80 & 3539.00 &       & 100\% \\
          & 11    & 32.0012 & 9411.20 &       & 100\% & 37.8657 & 1084.40 & 9208.20 &       & 100\% \\
          & 12    & 70.9765 & 12658.60 &       & 100\% & 37.6467 & 1104.00 & 11910.80 &       & 100\% \\
          & 13    & 710.0275 & 100078.40 &       & 100\% & 1679.7648 & 52401.40 & 287336.00 &       & 100\% \\
          & 14    & 4635.9384 & 287990.20 & 27.48\% & 60\% & 6433.5763 & 39467.20 & 192079.80 & 25.48\% & 40\% \\
          & 15    & 5741.0396 & 115401.20 & 7.12\% & 20\% & 3609.2785 & 11392.80 & 75087.00 & 10.55\% & 60\% \\\hline
    \end{tabular}}
  \label{table:2}
  \end{table}
\section{Branch-and-Cut solution algorithm\label{sec:Benders}}

In this section we describe the branch-and-cut  solution algorithm that we propose for solving  MSTN.  The special structure of MSTN, with disjoint domains for each set of variables - $x$ and $u$- and a bilinear objective function makes it possible to apply well-known Benders-like decomposition methods \cite{Benders,geoffrion}. This type of  well-known solution schemes have been widely applied to problems with two sets of structural decision variables, in which the subproblem that results when fixing one of the sets of variables can be \textit{efficiently} solved. Note that, as mentioned before, this requisite is satisfied in the case of  MSTN.

In order to warrant the convergence properties of the approach, we also apply reformulation techniques to the bilinear objective function. For a given spanning tree $\bar{x} \in \mathcal{ST}_G$, the ``optimal'' vertices and  distances of its associated MSTN, can be computed by solving the following convex subproblem:

\begin{align}
u(\bar x) = \min & \dsum_{e \in E}u_e \bar{x}_e\tag{${\rm PU_{\bar{x}}}$}\label{xmatheu0}\\
\text{s.t. } &  u \in \mathcal{U}\nonumber
\end{align}

As already mentioned, \eqref{xmatheu0} is a continuous SOC problem, which can be efficiently solved with on-the-shelf solvers. Note also that the number of $u$ variables in \eqref{xmatheu0} reduces to $n-1$, because only distances associated with the edges $e \in E$ with $\bar x_e=1$ need to be computed. Hence, (generalized) Benders decomposition is a suitable methodology for solving the MSTN problem.  The following result states  explicitly the form of the Benders cuts that allow to use particular solutions of  \eqref{xmatheu0} to solve MSTN.

\begin{thm}
Let $\bar{x} \in \mathcal{ST}_G$ and $u(\bar x)$ its associated \eqref{xmatheu0} solution. Then,
$$ \Theta \geq u(\bar x) +  \dsum_{e: \bar x_e=1} \widehat{U}_e (x_e-1)  +  \dsum_{e: \bar x_e=0} \widehat{u}_e x_e,$$
\noindent is a valid cut for MSTN, where, as before, $\Theta=\dsum_{e\in E}\theta_e$ with $\theta_e \ge 0$, $e\in E$; and $\widehat{U}_e$ and $\widehat{u}_e$ are upper bounds on the maximum and minimum values of the distance of edge $e$, respectively (for instance, the ones introduced after Example \ref{ex:1}).

\end{thm}
\begin{proof}

Let us consider the following equivalent reformulation of \eqref{xmatheu0} based on the Mckormick linearization of the bilinear terms of the objective function in the original MSTN formulation:

\begin{align}
u(\bar x) = \min & \dsum_{e \in E} \theta_e \nonumber\\
\text{s.t. } &  \theta_e \geq u_e + \widehat{U}_e (\bar{x}_e-1),\qquad\qquad e\in E\tag{${\rm RPU_{\bar{x}}}$}\label{refo-PU}\\
& \theta_e \geq \widehat{u}_e \bar x_e, \qquad\qquad\qquad\qquad\qquad e\in E\nonumber\\
 &  u \in \mathcal{U}.\nonumber
\end{align}
Note that the reformulation \eqref{refo-PU} is a convex optimization problem, and  Slater condition holds \cite{Slater50}. Hence, (necessary and sufficient) optimality conditions can be derived from the following Lagrangean function associated with \eqref{xmatheu0}:
$$
L(\bar x,  \theta, u; \lambda, \mu, \nu) = \dsum_{e \in E}  \theta_e - \dsum_{e \in E} \lambda_e ( \theta_e-u_e+\widehat{U}_e(1-\bar x_e)) - \dsum_{e \in E} \mu_e ( \theta_e-\widehat{u}_e\bar x_e) + \nu^t G(u),
$$
where $G(u) \leq 0$ are the constraints (only involving $u$-variables) defining $\mathcal{U}$.

Let $ \theta^*_e$, $u^*_e$, $e\in E$, be an optimal solution to \eqref{refo-PU} and $\lambda^*$, $\mu^*$ and $\nu^*$ the associated optimal multipliers.
Then, $\lambda^*$ and $\mu^*$ must satisfy:
\begin{equation}\label{eq:lagr-mult}
1 - \lambda^*_e - \mu^*_e = 0, \quad \forall e \in E,
\end{equation}
together with the complementary slackness constraints:
\begin{align*}
\lambda^*_e ( \theta^*_e-u^*_e+\widehat{U}_e(1-\bar x_e)) =0, &\quad  \forall e \in E,\\
\mu^*_e (\theta^*_e-\widehat{u}_e\bar x_e) =0, &\quad \forall e \in E.
\end{align*}
If $\bar x_e=1$, then  $\mu^*_e=0$, since $u^*\ge \hat u$ and $\theta^*_e\ge u^*_e$;  and by \eqref{eq:lagr-mult}, $\lambda^*_e=1$.
Besides, if $\bar x_e=0$, since  $u^*_e < \widehat{U}_e$, then $\theta_e^*=0$ and $\lambda^*_e=0$.
Thus, we conclude that:
\begin{equation}
\lambda^*=\bar x_e \text{ and } \mu_e^*=1-\bar x_e, \quad \forall  e\in E.\label{lambda-mu}
\end{equation}

On the other hand, since $u(x)=\Theta=\dsum_{e\in E}\theta_e=\max_{\lambda\geq0, \mu\geq0} \min_{\theta, u}L(x,  \theta, u; \lambda, \mu, \nu)$  also holds for any ${x} \in \mathcal{ST}_G$, we have that

\begin{align*}
\Theta &\; \geq \min_{\theta, u} L(\bar x,  \theta,u; \lambda^*, \mu^*, \nu^*) \\
 =&\; \dsum_{e \in E}  \theta^*_e - \dsum_{e \in E} \lambda^*_e ( \theta^*_e-u^*_e+\widehat{U}_e(1-\bar x_e)) - \dsum_{e \in E} \mu^*_e ( \theta^*_e-\widehat{u}_e \bar x_e) + {\nu^*}^t G(u^*) \\
=& \;\dsum_{e \in E}  \theta^*_e - \dsum_{e \in E} \lambda^*_e ( \theta^*_e-u^*_e+\widehat{U}_e(1- x_e)) - \dsum_{e \in E} \mu^*_e ( \theta^*_e-\widehat{u}_e  x_e) + {\nu^*}^t G(u^*)\\
& - \dsum_{e \in E} \lambda^*_e (\widehat{U}_e(1-\bar x_e)) + \dsum_{e \in E} \lambda^*_e (\widehat{U}_e(1- x_e)) - \dsum_{e \in E} \mu^*_e (\widehat{u}_e x_e)+\dsum_{e \in E} \mu^*_e (\widehat{u}_e \bar x_e)\\
=& \;u(\bar x) + \dsum_{e\in E} \lambda_e^* \widehat{U}_e (x_e-\bar x_e) +  \dsum_{e \in E} \mu^*_e \widehat{u}_e (x_e-\bar x_e)\\
=& \;u(\bar x) + \dsum_{e\in E: \bar x_e=1}  \widehat{U}_e (x_e-1) +  \dsum_{e \in E: \bar x_e=0}  \widehat{u}_e x_e.
 \end{align*}
This concludes the proof.
\end{proof}

Note that, by construction, the above generalized Benders cuts imply that, we can compare the value of the subproblem \eqref{xmatheu0} associated with a given spanning tree $\bar x \in \ST_G$, $u(\bar x)$, with the value of the  subproblem $({\rm RPU_{{x}}})$ associated with a different spanning tree $x \in \ST_G$, $u(x)$. In particular, if there exist $e_1, e_2\in E$ with $\bar x_{e_1}=1$ and $x_{e_1}=0$, and $\bar x_{e_2}=0$ and $x_{e_2}=1$, then the value of $u(x)$ is at least $u(\bar x) - \widehat{U}_{e_1} + \widehat{u}_{e_2}$. In other words, the difference between the values of the two subproblems is bounded by the maximum amount that can be saved (in the cost function) by removing $e_1$, plus the minimum gain that can be attained by adding $e_2$.
Therefore, the relaxed master problem at the $K$-th iteration of the  row-generation solution algorithm can be stated as:

\begin{align}
\Theta^* = \min \qquad &  \Theta \nonumber\\
    & \hspace*{-1cm}  \Theta \geq u(\bar x^k) +  \dsum_{e: \bar x^k_e=1} \widehat{U}_e (x_e - 1)  +  \dsum_{e: \bar x^k_e=0} \widehat{u}_e x_e, \; k=1, \ldots, K,\label{xmatheu000}\\
 & x \in \mathcal{ST}_G.\nonumber
\end{align}

The reader may note that the cuts (\ref{xmatheu000}) can be interpreted as some form of lifting of the surrogated McCorminck inequalities (\ref{linz}), after projecting out the $u$ variables in formulation (\ref{rl-mstn}).

Using the above cuts algorithmically, gives rise to the solution scheme described in Algorithm \ref{alg:benders}:
\medskip

\begin{algorithm}[H]
\SetKwInOut{Input}{Initialization}\SetKwInOut{Output}{output}

 \Input{Let $x^0\in \mathcal{ST}_G$ be an initial  solution and $\varepsilon$ a given threshold value.\\
 Set $LB=0$, $UB=+\infty$, $\bar x=x^0$.}

 \While{$|UB-LB|>\varepsilon$}{
 \begin{enumerate}
\item Solve \eqref{xmatheu0} for $\overline x$ to get $u(\bar{x})$.
\item Add the cut $ \Theta \geq u(\bar x) +  \dsum_{e: \bar x_e=1} \widehat{U}_e (x_e - 1)  +  \dsum_{e: \bar x_e=0} \widehat{u}_e x_e$ to the current master problem.
\item Obtain the optimal value $\bar{\Theta}$ to the current master problem, and its associated solution $\bar x$.
\item Update $LB=\max\{LB, \bar{\Theta}\}$ and $UB=\min\{UB, \sum_{e\in E}u(\bar{x})_e\bar{x}_e\}$
\end{enumerate}
}

 \caption{Decomposition Algorithm for solving MSTN.\label{alg:benders}}
\end{algorithm}

The stopping criterion is that the gap between the upper and lower bound does not exceed the fixed threshold value $\varepsilon$.
\medskip

\begin{thm}
The decomposition-based solution scheme of Algorithm \ref{alg:benders} terminates in a finite number of steps (for any given $\varepsilon\geq 0$). Furthermore, if $\varepsilon \le \min\{ \tilde{U}_{e_1}-\tilde{u}_{e_2}\ge 0: e_1\neq e_2 \in E\}$, it outputs an optimal MSTN.
\end{thm}
\begin{proof}
The the finiteness of the number of underlying spanning trees  of $\ST_G$, the convexity of \eqref{xmatheu0} for any $\overline x \in \ST_G$,  and the linear separability of the problem assure the result by applying Theorem 2.4 in \cite{geoffrion}.
\end{proof}

To avoid the enumeration of all spanning trees of $G$, and to reduce the number of iterations, several recipes can be applied. One of them is to start with a non-empty set of cuts which  give a suitable initial representation of the lower envelope of $\Theta$.
Hence, if $\overline{\mathcal{ST}_G}$ denotes the set of trees associated with the current set of constraints \eqref{xmatheu000}, the representation we use for the master problem is:

\begin{align}
\min \qquad & \dsum_{e\in E}\theta_e \label{aaxmatheu00} \\
\text{s.t. }    & \dsum_{e\in E}\theta_e \geq u(\bar x) +  \dsum_{e: \bar x_e=1} \widehat{U}_e (x_e - 1)  +  \dsum_{e: \bar x_e=0} \widehat{u}_e x_e, \forall \bar x \in \overline{\ST_G},\label{aaxmatheu000}\\
 & \theta_e\geq \widetilde{u}_e x_e, \quad e\in E,\nonumber\\
 & x \in \ST_G.\nonumber
\end{align}

Given that the master problem exhibits a combinatorial nature, the performance of a Benders-like algorithm can be improved by embedding the cut generation mechanism within a branch-and-cut scheme. This is the current trend nowadays \cite{Fischetti1,Fischetti2}. This requires to separate the optimality cuts in addition to any other generated cuts, at the nodes of the enumeration tree. Note that this approach is also valid in our case, as the cuts \eqref{aaxmatheu000} are also valid if $\bar x$ is the solution to a linear programming relaxation of a valid MST formulation.

\subsection{Computational Experiments \label{sec:Experiments}}

The proposed decomposition approach has been tested over the same set of benchmark instances used to compared the compact formulations. Based on the results obtained in such a comparison, and also to take advantage of the possibility of adding dynamically violated SECs within the branch-\&-cut, we combine the decomposition approach with the classical SEC representation \ref{sec-mstn}. In addition to the average statistics reported in the previous tables (CPU, \#SECs, \#Nodes, GAPs, and \%Solved), we also report now the average number of Benders' type cuts, \#BendersCuts, and the gap after the exploration of the root node of the branch-\&-cut tree, \%GAP$_0$. Average results for the 4 scenarios are reported in Tables \ref{table:3} and \ref{table:4}.\\

As can be seen, the computing times required by the decomposition approach are smaller than those obtained with the MINLP formulations for the small size radii scenario and also in the small-medium size radii scenario for the 3D case. However, the results obtained for the medium-large and large size scenarios reveal that the MINLP formulations have a better performance than the decomposition scheme. Note that the cuts induced by our approach depends of the available upper and lower bounds on the lengths  of the edges in the graph. These bounds are tight for the small size radii scenarios, but far from being a representative value of the actual length of the edge in the remaining scenarios. Hence, a large number of cuts are needed to certify optimality of the solution in these cases.

\begin{table}[hbpt]
  \caption{Average results for the decomposition approach for $\R^2$ instances.}
  {\scriptsize  \begin{tabular}{|r|r|rrrrrrr|}\hline
             $r$  & $n$      & CPU & \#SEC & \#BendersCuts & \#NodesB\%B & \%GAP$_0$ & \%GAP & \%Solved \\\hline
  \multirow{12}{*}{1} &     5     & 0.0065 & 1.20  & 0.20  & 0.00  & 5.45\% &       & 100\% \\
&     6     & 0.0196 & 3.60  & 2.40  & 10.40 & 18.99\% &       & 100\% \\
&     7     & 0.0328 & 5.60  & 4.00  & 22.80 & 12.07\% &       & 100\% \\
&     8     & 0.0347 & 3.60  & 3.80  & 23.40 & 15.41\% &       & 100\% \\
&     9     & 0.0646 & 12.80 & 7.60  & 64.60 & 18.79\% &       & 100\% \\
 &    10    & 0.1796 & 26.60 & 23.40 & 180.00 & 20.06\% &       & 100\% \\
 &    11    & 0.5341 & 116.60 & 68.40 & 950.60 & 28.48\% &       & 100\% \\
 &    12    & 0.6484 & 213.20 & 71.80 & 1129.00 & 32.67\% &       & 100\% \\
 &    13    & 1.5531 & 246.20 & 167.60 & 2573.80 & 37.76\% &       & 100\% \\
 &    14    & 1.6703 & 300.60 & 177.00 & 2204.20 & 32.39\% &       & 100\% \\
 &    15    & 45.3193 & 1016.40 & 1637.40 & 23077.60 & 47.74\% &       & 100\% \\
 &    20    & 333.5085 & 1628.60 & 3721.80 & 59876.80 & 39.75\% &       & 100\% \\\hline
  \multirow{11}{*}{2} &     5     & 0.0464 & 4.20  & 6.40  & 25.60 & 29.32\% &       & 100\% \\
&     6     & 0.0730 & 6.70  & 11.40 & 50.90 & 24.67\% &       & 100\% \\
&     7     & 0.0678 & 12.20 & 10.80 & 78.20 & 28.19\% &       & 100\% \\
&     8     & 0.2743 & 21.60 & 43.60 & 311.60 & 41.06\% &       & 100\% \\
&     9     & 0.3111 & 55.20 & 46.80 & 492.20 & 30.63\% &       & 100\% \\
 &    10    & 0.4646 & 78.00 & 66.60 & 721.40 & 33.22\% &       & 100\% \\
 &    11    & 1.3472 & 245.40 & 167.80 & 2382.60 & 35.11\% &       & 100\% \\
 &    12    & 160.8519 & 864.80 & 3027.00 & 36569.60 & 61.72\% &       & 100\% \\
 &    13    & 326.1787 & 1598.20 & 2800.40 & 50047.80 & 50.47\% &       & 100\% \\
 &    14    & 226.5067 & 2024.20 & 6463.60 & 96243.00 & 43.07\% &       & 100\% \\
 &    15    & 5824.7652 & 8023.00 & 18775.80 & 284590.80 & 73.80\% & 3.76\% & 20\% \\\hline
  \multirow{11}{*}{3} &     5     & 0.1152 & 3.80  & 5.80  & 24.40 & 27.67\% &       & 100\% \\
&     7     & 0.4851 & 58.10 & 93.50 & 712.60 & 50.67\% &       & 100\% \\
&     8     & 3.2475 & 158.20 & 526.80 & 3963.60 & 59.22\% &       & 100\% \\
&     9     & 17.3417 & 521.00 & 1492.40 & 14560.00 & 67.32\% &       & 100\% \\
 &    10    & 5.8312 & 226.20 & 595.00 & 5933.40 & 50.17\% &       & 100\% \\
 &    11    & 2603.6210 & 4308.40 & 12569.00 & 168712.00 & 75.77\% & 40.36\% & 80\% \\
 &    12    &  $>$7200 & 5223.40 & 23172.40 & 275986.80 & 81.98\% & 22.01\% & 0\% \\
 &    13    &  $>$7200 & 7191.60 & 20230.60 & 282031.60 & 85.37\% & 20.33\% & 0\% \\
 &    14    &  $>$7200 & 15425.00 & 14481.80 & 311567.60 & 90.64\% & 53.59\% & 0\% \\
 &    15    &  $>$7200 & 11379.40 & 13846.80 & 310549.80 & 83.69\% & 35.16\% & 0\% \\\hline
  \multirow{11}{*}{4} &     5     & 0.0476& 3.80  & 5.80  & 24.20 & 33.07\% &       & 100\% \\
&     6     & 0.3993 & 36.20 & 83.80 & 428.60 & 56.12\% &       & 100\% \\
&     8     & 2.9985 & 187.20 & 424.40 & 3055.80 & 62.99\% &       & 100\% \\
&     9     & 53.7040 & 418.00 & 2631.80 & 23586.40 & 67.46\% &       & 100\% \\
 &    10    & 1013.3837 & 1444.00 & 7611.00 & 72987.20 & 82.73\% &       & 100\% \\
 &    11    & 4256.8194 & 4636.60 & 16430.60 & 204272.20 & 84.16\% & 30.36\% & 60\% \\
 &    12    & 6232.3367 & 6569.80 & 20400.00 & 250014.00 & 77.37\% & 16.60\% & 20\% \\
 &    13    & $>$7200 & 8218.80 & 19321.40 & 299586.40 & 85.78\% & 29.58\% & 0\% \\
 &    14    &  $>$7200 & 13880.00 & 13080.80 & 336546.40 & 93.16\% & 71.25\% & 0\% \\
 &    15    &  $>$7200 & 16128.60 & 12538.00 & 326406.20 & 94.80\% & 50.14\% & 0\% \\\hline
              \end{tabular}}
  \label{table:3}
  \end{table}

  \begin{table}[hbpt]
  \caption{Average results for decomposition approach  for $\R^3$ instances.}
  {\scriptsize  \begin{tabular}{|r|r|rrrrrrr|}\hline
             $r$  & $n$      & CPU & \#SEC & \#BendersCuts & \#NodesB\%B & \%GAP$_0$ & \%GAP & \%Solved \\\hline
\multirow{12}{*}{1}&       5     & 0.0063 & 0.80  & 0.00  & 0.00  & 2.28\% &       & 100\% \\
&     6     & 0.0125 & 1.60  & 0.60  & 0.00  & 4.53\% &       & 100\% \\
&     7     & 0.0138 & 2.20  & 1.80  & 9.80  & 9.31\% &       & 100\% \\
&     8     & 0.0445 & 3.60  & 3.40  & 18.80 & 12.49\% &       & 100\% \\
 &    9     & 0.0573 & 6.00  & 5.80  & 28.60 & 9.66\% &       & 100\% \\
 &    10    & 0.0883 & 9.60  & 9.20  & 72.60 & 8.28\% &       & 100\% \\
 &    11    & 0.2478 & 30.20 & 17.20 & 162.00 & 17.26\% &       & 100\% \\
 &    12    & 0.2455 & 59.00 & 25.00 & 314.20 & 16.00\% &       & 100\% \\
 &    13    & 0.8280 & 87.60 & 85.20 & 1035.40 & 17.73\% &       & 100\% \\
 &    14    & 1.1512 & 194.80 & 95.20 & 1535.20 & 11.92\% &       & 100\% \\
 &    15    & 1.7121 & 264.00 & 130.20 & 1761.60 & 18.39\% &       & 100\% \\
  &    20    & 8.2175 & 702.20 & 377.80 & 7398.60 & 16.68\% &       & 100\% \\\hline
\multirow{11}{*}{2}&     5     & 0.0218 & 3.80  & 2.40  & 7.80  & 12.34\% &       & 100\% \\
&     6     & 0.0292 & 2.40  & 3.30  & 12.50 & 8.57\% &       & 100\% \\
&     7     & 0.0388 & 4.20  & 4.60  & 18.80 & 18.30\% &       & 100\% \\
 &    8     & 0.2397 & 24.00 & 33.40 & 247.00 & 23.13\% &       & 100\% \\
 &    9     & 0.2389 & 26.00 & 32.60 & 304.00 & 18.73\% &       & 100\% \\
 &    10    & 0.2859 & 50.80 & 33.00 & 398.00 & 12.74\% &       & 100\% \\
 &    11    & 0.5181 & 58.00 & 57.80 & 555.80 & 20.94\% &       & 100\% \\
 &    12    & 4.8255 & 263.20 & 369.80 & 5574.80 & 28.77\% &       & 100\% \\
 &    13    & 5.6111 & 498.60 & 635.80 & 9576.20 & 28.87\% &       & 100\% \\
  &    14    & 11.3739 & 1388.00 & 1459.40 & 32630.40 & 27.78\% &       & 100\% \\
&     15    & 35.4121 & 1873.00 & 2982.00 & 67628.20 & 33.80\% &       & 100\% \\\hline
\multirow{11}{*}{3}&    5&   0.0281 & 2.80  & 2.60  & 10.00 & 16.98\% &       & 100\% \\
 &6 &   0.2437 & 26.80 & 43.80 & 276.40 & 29.17\% &       & 100\% \\
 &7 &   0.2725 & 39.60 & 42.60 & 348.20 & 38.34\% &       & 100\% \\
 &8 &   1.5945 & 131.40 & 235.40 & 1915.20 & 49.20\% &       & 100\% \\
 &9 &   3.9492 & 292.40 & 1025.80 & 9022.00 & 45.81\% &       & 100\% \\
&10&   2.5790 & 313.00 & 272.40 & 3468.00 & 26.01\% &       & 100\% \\
&11 & 55.9248 & 689.40 & 1979.60 & 26140.60 & 42.86\% &       & 100\% \\
&12 &1258.5048 & 2060.40 & 8294.40 & 130089.80 & 47.82\% &       & 100\% \\
&13 &    3005.2253 & 5083.40 & 10824.20 & 212760.60 & 44.99\% & 3.81\% & 60\% \\
&14 &   $>$7200 & 9029.40 & 15154.60 & 288000.20 & 53.37\% & 17.53\% & 0\% \\
&15 &    1751.1580 & 9504.80 & 10049.00 & 243900.00 & 40.75\% &       & 100\%\\\hline
\multirow{11}{*}{4} &    5     & 0.0312 & 3.00  & 3.60  & 16.60 & 19.69\% &       & 100\% \\
 &    6     & 0.1750 & 13.80 & 29.60 & 122.20 & 27.05\% &       & 100\% \\
  &   7     & 0.6724 & 54.80 & 94.20 &  543.60 & 22.12\% &       & 100\% \\
 &    8     & 1.6626 & 162.80 & 218.80 & 1898.40 & 46.48\% &       & 100\% \\
 &    9     & 9.5678 & 326.60 & 916.20 & 8138.60 & 45.42\% &       & 100\% \\
 &    10    & 22.7335 & 576.60 & 1450.40 & 17267.40 & 47.61\% &       & 100\% \\
 &    11    & 107.0304 & 1037.60 & 3051.40 & 40153.60 & 50.56\% &       & 100\% \\
 &   12    & 1005.8061 & 1904.80 & 6533.00 & 99639.80 & 50.15\% &       & 100\% \\
  &  13    & 999.9207 & 5066.20 & 12905.60 & 211964.00 & 50.13\% &       & 100\% \\
   & 14    & 7200.3120 & 9951.60 & 14550.00 & 285772.40 & 70.62\% & 30.61\% & 0\% \\
    &15    & 6123.5383 & 12659.40 & 12203.20 & 266014.20 & 55.75\% & 16.35\% & 20\% \\\hline
                  \end{tabular}}
  \label{table:4}
  \end{table}

\section{A MathHeuristic for MSTN\label{sec:Mathheuristic}}
The results of the computational experiments section indicate that MSTN instances with up to less than 15 vertices can be optimally solved  within the allowed time limit, but as the sizes of the instances increase the computing times become prohibitive.
Below we present a mathheuristic alternative to obtain near-optimal solutions to larger MSTN instances. The main idea under the proposed algorithm is based on the observation that  the problem is a biconvex problem, since fixing any of the set of variables the problem becomes an efficiently solvable optimization problem (in case $x$ is fixed, the problem is a continuous SOCP, while if $u$ is fixed, the problem is a standard MST problem).

The mathheuristic consists of two embedded loops.The outer loop is a multistart procedure. The input of each iteration in this loop is a spanning tree, which will be used in the initial iteration of the inner loop. The number of iterations of the outer loop is a parameter related to the initial spanning tree generation mechanism that we use, which will be explained later on.

The rationale of the inner loop is to alternate in solving subproblems in the solution spaces of the two main sets of variables ($x$ and $u$).  We proceed iteratively, and each iteration consists of solving a pair  of subproblems, one in each space of variables. When solving the subproblem in one solution space we fix the values of the variables of the other space. 

Formally, let $(P_{\bar{x}u})$ and $(P_{{x}\bar{u}})$ respectively denote the subproblems of the generic MSTN formulation ${\rm P_{xu}}$ of Section \ref{sec:Formulations}, when $\bar{x}$ and $\bar{u}$  are fixed. That is,

\begin{minipage}[H]{0.4\textwidth}
\begin{align}
\min & \dsum_{e \in E}u_e \bar{x}_e\tag{${\rm PU_{\bar{x}}}$}\label{xmatheu0}\\
\text{s.t. } &  u \in \mathcal{U}\nonumber
\end{align}
\end{minipage}
\begin{minipage}[H]{0.4\textwidth}
\begin{align}
\text{and     }\qquad\min & \dsum_{e \in E}\bar{u}_e x_e\tag{${\rm PX_{\bar{u}}}$}\label{umatheu0}\\
\text{s.t. } & x \in \mathcal{ST}_G.\nonumber
\end{align}
\end{minipage}

\bigskip

Figure \ref{fig:Flowchart} shows a flowchart of the inner loop of the mathheuristic. We start with a given spanning tree $T^0$ associated with a solution $x^0$.
In the $k$-th iteration, we compute the distances $u(x^k)$ in the current  tree $T^k$ and update the vector $\bar{u}^{k+1}$ according to $\bar{u}^k$ and $u(x^k)$.
In the first iteration we use the distance lower bounds $\bar{u}^0=\widetilde{u}$. At each iteration $k>0$ we first solve problem $({\rm PX_{\bar{u}^k}})$ and then compute the vertices distances $u(x^k)$ in its optimal tree $T^k$, by solving $({\rm PU_{x^k}})$. All components  $\bar{u}^k_e$ associated with edges  $e\in T^k$ are updated to the corresponding component of the distances vector $u(x^k)$. The remaining components remain unchanged. The procedure terminates when two consecutive iterations produce the same tree or a maximum number of iteration is attained.

\begin{figure}[h]
\tikzstyle{block} = [rectangle, draw,text width=3.5cm, text centered, rounded corners, minimum height=4em]
\tikzstyle{block2} = [rectangle, draw,text width=4.5cm, text centered, rounded corners, minimum height=4em]
\tikzstyle{line} = [draw, -latex']
\tikzstyle{cloud1} = [draw, circle, fill=gray!10, text width=1.15cm, text centered]
 \tikzstyle{cloud2} = [draw, ellipse, minimum height=1.7em, text width=1.15cm, text centered]

{\small\begin{tikzpicture}[scale=0.5,node distance = 3cm, every text node part/.style={align=center}]
    \node[block] (px) {Solve $({\rm PX_{\bar{u}^k}})$ \\ $\downarrow$\\$T^k$ associated with $x^k$};
    \node [block, below of=px] (dist) {Solve $({\rm PU_{x^k}})$ \\ $\downarrow$\\ obtain distances $u(x^k)$};
        \node [cloud1, left of=dist,node distance=4cm] (init) {{\scriptsize$k\leftarrow 0$} \\ $T^0$, $x^0$};
     \node [block2, right of=dist, node distance=4.5cm] (update) {Update $\bar u^k$: \bigskip\\
   $\bar u^{k+1}_ e = \left\{\begin{array}{cl} u_e(x^k) & \mbox{if $e\in T^k$,}\\\bar u^{k}_ e & \mbox{otherwise}\end{array}\right.$};

    \node [cloud2, right of=px, node distance=4.5cm] (k) {\scriptsize$k \leftarrow k+1$};
    \path [line,dashed] (init) -- (dist);
    \path [line] (px) -- (dist);
    \path [line] (dist) -- (update);
    \path [line] (update) --  (k);
    \path [line] (k) -- (px);
\end{tikzpicture}}

        \caption{Flowchart of the inner loop of the mathheuristic}\label{fig:Flowchart}
 \end{figure}
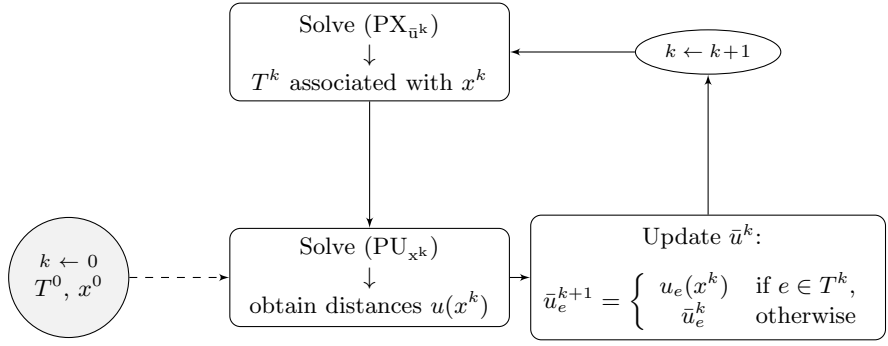%

For the sake of analyzing the quality of solutions obtained with the mathheuristic we introduce the notion of \textit{partial optimal MSTN} adapting the notation in \cite{WH} for the general case of minimizing a non-separable function subject to disjoint constraints.

\begin{defn}[Partial Optimum MSTN]
Let $\bar x \in \ST_G$ and $\bar u \in \mathcal{U}$. $(\bar x, \bar u)$ is said a partial optimum MSTN if:
$$
\dsum_{e\in E} \bar x_e \bar u_e \leq \dsum_{e\in E} x_e \bar u_e \quad and \quad \dsum_{e\in E} \bar x_e \bar u_e \leq \dsum_{e\in E} \bar x_e u_e
$$
for all $x \in \ST_G$ and $u \in \mathcal{U}$.
\end{defn}

Observe that a partial optimum MSTN $(\bar x, \bar u)$ implies that $\bar x$ is a MST for the weights $\bar u$ and that $\bar u$ are the optimal distances with respect to $\bar x$. The following result states the partial optimality of the solutions generated by the proposed mathheuristic.
\begin{thm}
The sequence of objective values produced at the inner loop of the mathheuristic, corresponding to a given initial solution,  converges monotonically to a partial optimum MSTN.
\end{thm}
\begin{proof}
Let $f(x,u)= \dsum_{e\in E} x_e u_e$  denote the objective function value associated with a given solution $x \in \mathcal{ST}_G$, $u \in \mathcal{U}$. Let also $ x^1, \ldots, x^k\in \ST_G$ and $u^1, \ldots, u^k \in \mathcal{U}$ be the solutions obtained in the first $k$ steps of the alternate convex search for a given initial solution.

Observe that in the mathheuristic, for $ u^j$ given, $ x^{j+1}$ is obtained by solving $({\rm PX_{\bar{u}}})$ with weights $\bar u=u^j$. Hence, 
$$
 \dsum_{e\in E} x^{j+1}_e\; u_e^j \leq \dsum_{e\in E}  x_e \; \bar u_e^j, \forall x \in \ST_G.
$$

Next, solving $({\rm PU_{\bar x}})$ with $\bar x= x^{j+1}$, one obtains $u(x^{j+1})$ and then $u^{j+1}$ with:

$$
\dsum_{e\in E} x^{j+1}_e\; u(x^{j+1})_e=\dsum_{e\in E} x^{j+1}_e\; u^{j+1}_e \leq \dsum_{e\in E}  x^{j+1}_e\; u_e, \forall u \in \mathcal{U}.
$$

Hence, $f(x^{k}, u^{k}) \geq f(x^k, u(x^k))\geq f(x^{k+1}, u^{k+1})$, so the sequence $\{f(x^j, u^j)\}_{j\in \mathbb{Z}_+}$ is monotonically non-increasing. Thus, since $f(x,u)\geq 0$ for all $x\in \ST_G$ and $u\in \mathcal{U}$, the sequence of objective values converges.\\

Let $\Theta^* = \lim\limits_{j\rightarrow\infty} f(x^j,u^j)$ and $x^*\in \ST_G$, $u^*\in \mathcal{U}$ such that $f(x^*, u^*)=\Theta^*$.
Since $\ST_G$ and $\mathcal{U}$ are closed sets and $f$ is continuous, we have that taking limits:
$$
\Theta^* = \dsum_{e\in E} x^*_e u^*_e \leq \dsum_{e\in E} x_e u^*_e \quad and \quad \Theta^* = \dsum_{e\in E} x^*_e u^*_e \leq \dsum_{e\in E} x^*_e u_e.
$$
Thus, $(x^*, u^*)$ is a partial optimum MSTN.
\end{proof}

Since only partial optimality of the solutions is assured at the end of each inner loop, it is possible that the mathheuristic gets trapped at a local optimum. Hence we have incorporated a multistart outer loop to allow escaping from local optimal. Note that the mathheuristic becomes an exact solution method if all possible spanning trees are considered as initial solutions. However, complete enumeration is prohibitive, even if the number of potential MSTs is finite (despite using varying weights). On the other hand, we have observed that $(i)$ the mathheuristic is sensitive to the provided initial feasible solution, and; $(ii)$ in many cases, a few changes over an initial standard MST with respect to the distances between the centers of the neighborhoods are enough to find an optimal  MSTN solution.
Hence, we generate the set of initial spanning trees for the multistart procedure with an adaptation of the method proposed in \cite{generatingst}, which is described in Algorithm \ref{alg:multistart}.
In principle, this method generates the whole set of spanning trees on a given graph (by increasing order values relative to a given weight vector). In our adaptation, we stop generating new spanning trees, when one of the following criteria is met: (1) a given number of MSTs has already been generated; or, (2) no improvement has been obtained, in the MSTNs obtained in the inner iterations, for a given number of outer iterations.

\begin{algorithm}[H]
\SetKwInOut{Input}{Initialization}\SetKwInOut{Output}{output}

 \Input{$u^0_{vw} = \|v-w\|$, $\forall v, w \in V$ and $T^0$ the MST with respect to $u^0$, $\mathcal{T}=\{T^0\}$.}

\For{$T\in \mathcal{T}$}{

 Let $e_1, \ldots, e_{n-1}$ be the edges of $T$.

\For{$i=1, \ldots, n-1$}{
Construct the MST with respect to $u^0$, $T_i$, such that $e_i$ does not belong to the tree but $e_1, \ldots, e_{i-1}$ are part of it. Let $c_i$ be the weight of $T_i$.
}

Choose $T' \in \{T_1, \ldots, T_{n-1}\}$ with $c(T^{'}) = \dmin_{i=1, \ldots, n} c_i$ and add it to $\mathcal{T}$.

}

 \caption{Initial solutions for the multistart procedure.\label{alg:multistart}}
\end{algorithm}

\vspace*{0.5cm}

A series of computational experiments have been performed to analyze the computing times and the quality of the solutions obtained with the overall heuristic. We report results based on two batteries of  benchmark instances. The first one is the same that was used in our previous experiments. Here the goal is to compare the quality of the solutions obtained by the exact and the heuristic methods. The second one contains larger size instances and the goal is to explore the limit of the mathheuristic. In the experiments we do not fix limits on the number of inner iterations but we set up the maximum number of trees generated (outer iterations) to $100\times |E|$. Table \ref{table:5-6} summarizes the obtained numerical results. We report average values of the computing times consumed the the mathheuristic (CPU) and the percentage deviation (\%Dev) with respect to the optimal (or best-known) solutions obtained with the exact approaches. Observe that the quality of the solutions is extremely good, as the maximum \%Dev obtained in all the experiments was $1.3086\%$. Furthermore, in most of the cases where the exact approaches did not prove the optimality of the best solution found, the heuristic produced a better solution. Indeed, many of the proven optimal solutions obtained with the other approaches, where also obtained with the mathheuristic. Moreover, in a few cases the mathheuristic gives slightly better solutions than those obtained by the exact methods which showed some precision difficulties caused by numerical instability.  Tables \ref{table:7} and \ref{table:8} show the results for the largest instances. We report, apart from the average computing times, the percentage deviations with respect to available lower (\%Dev LB) and upper bounds (\%Dev UB) for the optimal value of the MSTN. Lower bounds were calculated by computing the MST with respect to the original graph in which the edge lengths are given as  the minimum distance between the neighborhoods that contain the vertices of each edge, i.e.:
$$
\bar u_{e} = \min \{d(y_v, y_w): y_v \in \mathcal{N}_v, y_w\in \mathcal{N}_w\}, \quad \text{ for } e=\{v,w\} \in E.
$$
Upper bounds are computed as the optimal value of \eqref{xmatheu0}, when $\bar x$ is the standard MST. Finally, we also report the percentage of instances (out of $5$) in which the solution of the matheuristic coincides with the upper bound (i.e. the underlined MSTN equals the MST). As expected, the deviations with respect to the lower and upper bounds increases as the radii of the neighborhoods do. The same happens with the number of instances in which the solutions of the MSTN coincide with those of MST. In scenario 4, the instances with largest radii, the lower bounds are close to zero in most of the cases since almost all pairs of neighborhoods intersect, and several 100\% deviations were obtained.  The reader may observe that deviation with respect to lower bounds are few significative since these bounds are always rather far from the actual optimal solution. We would also like to emphasize that  computing times for the three-dimensional instances are slightly larger than those obtained for the planar instances, due to the number of variables of the problems \eqref{xmatheu0}, that must be iteratively solved in the inner loop of the algorithm.  However, the times do not seem to largely depend of the size of the neighborhoods.

\begin{table}[htbp]

  \caption{Average Results for the mathheuristic.}
   {\small
   \begin{center}\begin{tabular}{|r|r|cc|c|cc|}\cline{3-4}\cline{6-7}
\multicolumn{2}{c}{} & \multicolumn{2}{|c|}{2-dimensional instances}& & \multicolumn{2}{|c|}{3-dimensional instances}\\ \cline{1-4}\cline{6-7}
 $r$ &   $n$ & CPU & \%Dev &                                                         & {CPU} & {\%Dev} \\\cline{1-5}\cline{6-7}
\multirow{12}{*}{1} &     {5} & 0.1004 & 0.0000\% &                                  & 0.1594 & 0.0000\% \\
&     {6} & 0.2068 & 0.0000\%   &                                                    & 0.2200 & 0.0000\% \\
 &    {7} & 0.3368 & 0.0433\%   &                                                    & 0.3614 & 0.0001\% \\
 &    {8} & 0.5220 & 0.0000\%   &                                                    & 0.7036 & 0.0000\% \\
 &    {9} & 0.6982 & 0.0000\%   &                                                    & 0.6792 & 0.0195\% \\
 &    {10} & 1.2014 & 0.1768\%  &                                                    & 1.2254 & 0.0000\% \\
 &    {11} & 1.8868 & 0.2679\%  &                                                    & 2.1230 & 0.3749\% \\
 &    {12} & 2.4382 & 0.0000\%  &                                                    & 2.3078 & 0.0000\% \\
 &    {13} & 3.0136 & 0.1319\%  &                                                    & 4.1954 & 0.1223\% \\
&     {14} & 3.9986 & 0.1802\%  &                                                    & 4.0428 & 0.0527\% \\
&     {15} & 5.9238 & 0.3095\%  &                                                    & 5.4956 & 0.2659\% \\
&     {20} & 15.3978 & 0.2068\% &                                                    & 15.4622 & 0.0565\% \\\cline{1-5}\cline{6-7}
\multirow{11}{*}{2} &    {5} & 0.1788 & 0.0000\% &                                   & 0.2416 & 0.0001\% \\
 &    {6} & 0.2603 & 0.0011\%  &                                                     & 0.3098 & 0.0000\% \\
 &    {7} & 0.3972 & 0.1528\%  &                                                    & 0.5358 & 0.0000\% \\
 &    {8} & 0.8566 & 0.0000\%  &                                                    & 1.3224 & 0.0000\% \\
 &    {9} & 0.9240 & 0.6322\%  &                                                    & 0.9988 & 0.3318\% \\
 &    {10} & 1.4706 & 0.1666\% &                                                    & 1.6722 & 0.0296\% \\
 &    {11} & 2.0872 & 0.8081\% &                                                    & 2.5434 & 0.3964\% \\
 &    {12} & 3.1428 & 0.0212\% &                                                    & 4.2852 & 0.2285\% \\
 &    {13} & 3.7266 & 0.5755\% &                                                    & 6.3750 & 0.3975\% \\
  &   {14} & 5.6144 & 0.5838\% &                                                    & 6.5618 & 0.0270\% \\
  &   {15} & 9.1994 & -0.0408\%&                                                    & 10.2092 & 0.3245\% \\\cline{1-5}\cline{6-7}
\multirow{11}{*}{3}  &   {5} & 0.1710 & 0.0000\% &                                  & 0.2370 & 0.0000\% \\
 &   {6} & 0.2134 & 0.0000\% &                                  & 0.6210 & 0.0000\% \\
  &   {7} & 0.5969 & 0.1360\%  &                                                    & 0.7737 & 0.0713\% \\
  &   {8} & 0.9008 & 0.1571\%  &                                                    & 1.3504 & 0.0271\% \\
  &   {9} & 1.3432 & 1.3086\%  &                                                    & 2.3226 & 0.7177\% \\
&     {10} & 1.8258 & 0.8340\% &                                                    & 2.6464 & 0.4596\% \\
&     {11} & 3.0670 & 0.1899\% &                                                    & 4.4142 & 1.1838\% \\
&     {12} & 4.3984 & 0.1122\% &                                                    & 5.2298 & 0.0581\% \\
 &    {13} & 4.9976 & 0.4673\% &                                                    & 7.1142 & 1.2851\% \\
 &    {14} & 6.7682 & -0.1210\%&                                                    & 10.2342 & -0.1614\% \\
 &    {15} & 8.2982 & -0.0949\%&                                                    & 11.2072 & 0.2390\% \\\cline{1-5}\cline{6-7}
\multirow{11}{*}{4} &    {5} & 0.1664 & 0.0000\% &                                  & 0.2738 & 0.0000\% \\
 &    {6} & 0.3942 & 0.1012\% &                                                     & 0.4942 & 0.5379\% \\
  &    {7} & 0.7893 & 0.0601\% &                                                     & 0.9942 & 0.1123\% \\
 &    {8} & 1.1640 & 0.0000\% &                                                     & 1.6256 & 0.0353\% \\
 &    {9} & 1.5462 & 0.7477\% &                                                     & 1.8514 & 0.4004\% \\
 &    {10} & 2.2468 & 1.1261\%&                                                     & 2.6576 & 1.3283\% \\
 &    {11} & 3.2060 & 0.7875\%&                                                     & 3.6996 & 0.6159\% \\
  &   {12} & 4.5152 & 0.2935\%&                                                     & 4.8816 & 0.1611\% \\
  &   {13} & 5.0992 & 0.7808\%&                                                     & 7.2430 & 1.0225\% \\
  &   {14} & 6.8126 & -0.1978\% &                                                   & 9.6768 & 0.6739\% \\
  &   {15} & 8.1124 & 0.0105\%&                                                     & 11.6100 & -0.2135\%\\\cline{1-5}\cline{6-7}
     \end{tabular}
     \end{center}
     }
   \label{table:5-6}
\end{table}

\begin{table}[htbp]
  \centering
  \caption{Average Results for the mathheuristic for large instances in the planar case.}
 {\small   \begin{tabular}{|r|r|rrrr|}\hline
  $r$ & $|V|$ & CPU & \%Dev LB & \% Dev UB & \% MST \\\hline
\multirow{12}{*}{1} &     {20} & 14.5532 & 23.0877\% & 0.1201\% & 40.00\% \\
&    {25} & 27.1624 & 27.6163\% & 0.2969\% & 40.00\% \\
&    {30} & 54.5254 & 27.8230\% & 0.3004\% & 40.00\% \\
&    {35} & 82.7320 & 28.8806\% & 0.1985\% & 40.00\% \\
&    {40} & 122.8916 & 28.7590\% & 0.3342\% & 40.00\% \\
&    {45} & 182.2026 & 38.7607\% & 0.1451\% & 80.00\% \\
&    {50} & 255.4392 & 43.0832\% & 0.0912\% & 80.00\% \\
&    {60} & 472.9626 & 40.6246\% & 0.2814\% & 20.00\% \\
&    {70} & 724.8468 & 43.4054\% & 0.1118\% & 80.00\% \\
&    {80} & 751.3728 & 47.7128\% & 0.3567\% & 40.00\% \\
&    {90} & 1064.7958 & 49.2007\% & 0.0000\% & 100.00\% \\
&    {100} & 1480.0034 & 53.4484\% & 0.1639\% & 80.00\% \\\hline
\multirow{12}{*}{2} &    {20} & 16.4950 & 62.3051\% & 1.3996\% & 0.00\% \\
 &   {25} & 31.6210 & 77.3769\% & 0.3444\% & 20.00\% \\
 &   {30} & 59.5594 & 77.6920\% & 1.5311\% & 0.00\% \\
 &   {35} & 87.8010 & 86.6972\% & 2.4308\% & 0.00\% \\
 &   {40} & 145.0846 & 87.3522\% & 1.2426\% & 40.00\% \\
 &   {45} & 192.4576 & 84.7788\% & 0.7022\% & 60.00\% \\
 &   {50} & 283.2516 & 91.5316\% & 1.0501\% & 40.00\% \\
 &   {60} & 525.9362 & 96.1926\% & 1.6971\% & 0.00\% \\
 &   {70} & 835.0496 & 96.2605\% & 0.8858\% & 20.00\% \\
 &   {80} & 779.3946 & 97.2727\% & 0.9087\% & 40.00\% \\
 &   {90} & 1122.9898 & 98.3883\% & 0.5728\% & 60.00\% \\
 &   {100} & 1548.9070 & 99.3069\% & 1.4232\% & 40.00\% \\\hline
\multirow{12}{*}{3} &    {20} & 16.0632 & 90.6985\% & 2.2212\% & 20.00\% \\
&    {25} & 32.1278 & 96.2322\% & 0.7643\% & 20.00\% \\
&    {30} & 65.7792 & 97.5944\% & 1.0350\% & 0.00\% \\
&    {35} & 90.1888 & 98.4009\% & 5.9840\% & 0.00\% \\
&    {40} & 137.5042 & 99.0318\% & 2.0271\% & 0.00\% \\
&    {45} & 198.4974 & 99.0682\% & 1.0427\% & 40.00\% \\
&    {50} & 268.2828 & 99.8648\% & 2.2477\% & 20.00\% \\
&    {60} & 502.3478 & 100.0000\% & 3.2364\% & 0.00\% \\
&    {70} & 816.0300 & 100.0000\% & 2.7085\% & 20.00\% \\
&    {80} & 756.5704 & 100.0000\% & 2.3165\% & 40.00\% \\
&    {90} & 1116.6500 & 100.0000\% & 1.8877\% & 40.00\% \\
&    {100} & 1530.6052 & 100.0000\% & 1.5370\% & 20.00\% \\\hline
\multirow{12}{*}{4} &    {20} & 16.4998 & 97.9307\% & 2.7959\% & 20.00\% \\
&    {25} & 33.8690 & 99.3203\% & 1.6366\% & 20.00\% \\
&    {30} & 61.1976 & 100.0000\% & 2.6932\% & 0.00\% \\
&    {35} & 89.4202 & 100.0000\% & 8.7080\% & 0.00\% \\
&    {40} & 146.1266 & 100.0000\% & 3.3380\% & 0.00\% \\
&    {45} & 213.8344 & 100.0000\% & 3.0796\% & 20.00\% \\
&    {50} & 282.9736 & 100.0000\% & 2.0663\% & 20.00\% \\
&    {60} & 486.8964 & 100.0000\% & 4.5859\% & 0.00\% \\
&    {70} & 763.0016 & 100.0000\% & 4.2135\% & 0.00\% \\
&    {80} & 748.1272 & 100.0000\% & 4.5767\% & 0.00\% \\
&    {90} & 1085.8690 & 100.0000\% & 3.2538\% & 20.00\% \\
&    {100} & 1668.2424 & 100.0000\% & 2.7675\% & 20.00\% \\\hline
        \end{tabular}}
       \label{table:7}
\end{table}

\begin{table}[htbp]
  \centering
  \caption{Average Results for the mathheuristic for large instances in the 3D case.}
{\small   \begin{tabular}{|r|r|rrrr|}\hline
  $r$ & $|V|$ & CPU & \%Dev LB & \% Dev UB & \% MST \\\hline
\multirow{12}{*}{1} &   20    & 14.6272 & 10.3986\% & 0.0467\% & 80.00\% \\
&    25    & 40.6772 & 13.0944\% & 0.0378\% & 60.00\% \\
&    30    & 69.8356 & 10.6289\% & 0.0006\% & 80.00\% \\
&    35    & 106.5134 & 11.2375\% & 0.1286\% & 60.00\% \\
&    40    & 175.6634 & 11.0897\% & 0.1831\% & 40.00\% \\
&    45    & 262.2358 & 15.1212\% & 0.0322\% & 80.00\% \\
&    50    & 370.5236 & 17.9594\% & 0.2323\% & 60.00\% \\
&    60    & 631.6412 & 14.9262\% & 0.0000\% & 100.00\% \\
&    70    & 1071.5590 & 18.0318\% & 0.1747\% & 60.00\% \\
&    80    & 1071.1360 & 17.2028\% & 0.1713\% & 60.00\% \\
&    90    & 1570.6312 & 17.1973\% & 0.0046\% & 80.00\% \\
&    100  & 2256.3462 & 20.5805\% & 0.1206\% & 60.00\% \\\hline
\multirow{12}{*}{2} &    20    & 24.0912 & 34.4738\% & 0.9106\% & 20.00\% \\
&    25    & 49.7172 & 47.0066\% & 0.4466\% & 20.00\% \\
&    30    & 81.0262 & 40.1495\% & 1.3887\% & 20.00\% \\
&    35    & 123.2108 & 45.9130\% & 0.4637\% & 60.00\% \\
&    40    & 211.2694 & 48.8337\% & 0.9941\% & 20.00\% \\
&    45    & 295.5366 & 52.4260\% & 0.2171\% & 60.00\% \\
&    50    & 401.4358 & 55.8653\% & 0.5822\% & 60.00\% \\
&    60    & 743.1540 & 61.8838\% & 0.2815\% & 60.00\% \\
&    70    & 1139.6448 & 68.2234\% & 0.7040\% & 40.00\% \\
&    80    & 1145.8188 & 69.4113\% & 0.4693\% & 40.00\% \\
&    90    & 1835.7320 & 71.7928\% & 0.5406\% & 40.00\% \\
&    100   & 2456.1402 & 77.1601\% & 0.1699\% & 60.00\% \\\hline
\multirow{12}{*}{3} &    20    & 24.9052 & 66.9737\% & 2.4841\% & 20.00\% \\
&    25    & 51.9204 & 76.8203\% & 2.8566\% & 0.00\% \\
&    30    & 83.2864 & 75.1517\% & 3.4033\% & 20.00\% \\
&    35    & 136.2574 & 83.2923\% & 0.8824\% & 40.00\% \\
&    40    & 207.1532 & 82.1425\% & 3.4419\% & 0.00\% \\
&    45    & 293.3924 & 85.7698\% & 1.1218\% & 20.00\% \\
&    50    & 431.9292 & 91.9528\% & 1.6269\% & 40.00\% \\
&    60    & 741.9330 & 96.3082\% & 2.9933\% & 20.00\% \\
&    70    & 1163.3446 & 97.8903\% & 2.2103\% & 0.00\% \\
&    80    & 1231.5932 & 97.5674\% & 0.9325\% & 40.00\% \\
&    90    & 1770.6206 & 98.3531\% & 1.5740\% & 20.00\% \\
&    100   & 2357.2434 & 98.5889\% & 2.7997\% & 20.00\% \\\hline
\multirow{12}{*}{4} &    20    & 24.4860 & 90.5059\% & 4.3812\% & 0.00\% \\
&    25    & 50.6444 & 93.6932\% & 2.9003\% & 0.00\% \\
&    30    & 84.3946 & 96.3750\% & 4.9004\% & 20.00\% \\
&    35    & 134.4824 & 97.3869\% & 2.5519\% & 20.00\% \\
&    40    & 213.2442 & 98.0207\% & 5.4728\% & 0.00\% \\
&    45    & 304.6368 & 99.5034\% & 1.9230\% & 0.00\% \\
&    50    & 415.3388 & 99.3344\% & 3.2609\% & 0.00\% \\
&    60    & 721.3308 & 99.9964\% & 2.7762\% & 20.00\% \\
&    70    & 1189.9664 & 100.0000\% & 3.0113\% & 0.00\% \\
&    80    & 1233.2842 & 100.0000\% & 2.1201\% & 20.00\% \\
&    90    & 1922.6220 & 100.0000\% & 2.3436\% & 0.00\% \\
&    100   & 2412.5672 & 100.0000\% & 2.9934\% & 20.00\%\\\hline
\end{tabular}}
       \label{table:8}
\end{table}

\section{Conclusions}

We analyzed the problem of finding minimum spanning trees with neighborhoods, where the neighborhoods are defined as SOC-representable objects and the lengths of the arcs in the graph are induced by a $\ell_q$ norm. Two MINLP formulations are provided whose differences come from the representation of the subtour elimination constraints. We propose a decomposition-based methodology to solve the problem based on the efficiency of solving SOCP problems. Furthermore, a new mathheuristic procedure is applied to solve the problem exploiting not only the SOC-representability of the neighborhoods but also that the MST problems are easily solvable. The results of an extensive computational experience are reported to compare all formulations and procedures provided throughout this paper.

\section*{Acknowledgements}

The first and third authors were partially supported by the project MTM2016-74983-C2-1-R (MINECO, Spain). The second author was partially supported by the project MTM2015-63779-R (MINECO, Spain).

\end{document}